\documentclass{article}

\usepackage{amsmath, amssymb, amsthm, color, enumerate, mathrsfs, graphicx}

\usepackage{color}
\definecolor{myurlcolor}{rgb}{0.1,0.1,0.8}
\definecolor{mylinkcolor}{rgb}{0.05,0.05,0.4}

\usepackage{hyperref}
\hypersetup{colorlinks,
	linkcolor=mylinkcolor,
	citecolor=mylinkcolor,
	urlcolor=myurlcolor}

\usepackage{tocloft}
\usepackage[numbers]{natbib}
\usepackage{natbibspacing}
\renewcommand{\bibfont}{\small}

\theoremstyle{theorem}
\newtheorem{thm}{Theorem}
\newtheorem*{theorem*}{Theorem}
\numberwithin{thm}{section}
\newtheorem{cor}[thm]{Corollary}
\newtheorem{lem}[thm]{Lemma}
\newtheorem{prop}[thm]{Proposition}

\theoremstyle{definition}
\newtheorem{defn}[thm]{Definition}
\newtheorem{example}[thm]{Example}
\newtheorem{rmk}[thm]{Remark}
\newtheorem{rmks}[thm]{Remarks}

\newcommand{\demph}[1]{\textbf{#1}} % emphasis for definitions
\newcommand{\cn}{\mathbf}

\newcommand{\bb}{\mathbb}
\newcommand{\R}{\bb{R}}

\newcommand{\xto}{\xrightarrow}
\newcommand{\st}{\, : \,}

\newcommand{\supp}{\mathrm{supp} \, }
\newcommand{\esssup}{\mathrm{ess \, sup}}
\newcommand{\essinf}{\mathrm{ess \, inf}}
\newcommand{\wext}[1]{\widehat{#1}} % normalization & extension of weighting
\newcommand{\Dmax}[2]{D_{\max}(#1, #2)}
\newcommand{\Dmx}[1]{D_{\max}(#1)}
\newcommand{\Hmax}[2]{H_{\max}(#1, #2)}
\newcommand{\Hmx}[1]{H_{\max}(#1)}
\newcommand{\glob}[1]{\textbf{\textit{#1}}}
\newcommand{\hard}[1]{#1}
\newcommand{\ip}[2]{\langle #1, #2 \rangle}
\newcommand{\Ip}[2]{\bigl\langle #1, #2 \bigr\rangle}
\newcommand{\cell}[4]{\put(#1,#2){\makebox(0,0)[#3]{#4}}}
\newcommand{\dH}{d_{\mathrm{H}}}

\renewcommand{\d}{\,\mathrm{d}}
\renewcommand{\vec}{\mathbf}
\renewcommand{\epsilon}{\varepsilon}

\title{The maximum entropy of a metric space}
\author{Tom Leinster%
\thanks{School of Mathematics, University of Edinburgh;
  Tom.Leinster@ed.ac.uk}
 \qquad 
Emily Roff%
\thanks{School of Mathematics, University of Edinburgh;
Emily.Roff@ed.ac.uk}}
\date{\vspace*{-1cm}}

\begin{document}
\sloppy
\maketitle

\begin{abstract}
We define a one-parameter family of entropies, each assigning a real number
to any probability measure on a compact metric space (or, more generally, a
compact Hausdorff space with a notion of similarity between points).  These
generalise the Shannon and R\'enyi entropies of information theory.

We prove that on any space $X$, there is a single probability measure
maximising all these entropies simultaneously.  Moreover, all the entropies
have the same maximum value: the \emph{maximum entropy} of $X$.  As $X$ is
scaled up, the maximum entropy grows, and its asymptotics determine
geometric information about $X$ including the volume and dimension.  And
the large-scale limit of the maximising measure itself provides an
answer to the question: what is the canonical measure on a metric space?

Primarily we work not with entropy itself but its exponential, which in its
finite form is already in use as a measure of biodiversity.  Our main
theorem was first proved in the finite case by Leinster and
Meckes~\cite{LeinsterMaximizing2016}.
\end{abstract}

\tableofcontents

%%%%%%%%%%%%%%%%% ORIGINS %%%%%%%%%%%%%%%%%

\section{Introduction}
\label{S_intro}

This paper introduces and explores a largely new invariant of compact
metric spaces: the maximum entropy.  Intuitively, this measures how much room
a probability distribution on the space has available to spread out.

Maximum entropy has several claims to importance.  First, it is the maximal
value of not just \emph{one} measure of entropy, but an \emph{uncountable
  infinity} of them.  It is a theorem, proved here, that these entropy
measures all have the same maximum.  

Second, the entropies concerned are already established in ecology, where
their exponentials are used as measures of biological
diversity~\cite{LeinsterMeasuring2012}.  Indeed, they have been applied to
ecological systems at all scales, from microbes~\cite{BCMV} and
plankton~\cite{JTYP} to fungi~\cite{VPDL}, plants~\cite{CMLT}, and large
mammals~\cite{BRBT}.  Relative to other diversity measures, they have been
found to improve inferences about the diversity of natural
systems~\cite{VPDL}.

Third, the exponential of maximum entropy---called maximum
diversity---plays a similar conceptual role for metric spaces as
cardinality does for sets.  In the special case of a finite space where all
distances are $\infty$, it is literally the cardinality, and in general, it
increases when the space is enlarged (either by adding new points or
increasing distances).

Fourth, unlike most geometric invariants, maximum entropy is `informative
under rescaling': the maximum entropy of a metric space $X$ does not
determine the maximum entropy of $tX$ for scale factors $t \neq 1$.
Maximum entropy therefore assigns to $X$ not just a single \emph{number},
but a \emph{function}, the maximum entropy of $tX$ as a function of $t$.
The asymptotics of this function turn out to determine the volume and
dimension of $X$---themselves geometric analogues of cardinality.

Finally, maximum diversity is in principle a known quantity in potential
theory, where it belongs to the family of Bessel capacities, although it
lies just outside the part of the family usually studied by potential
theorists (Remark~\ref{rmks:Dmax-euc}(\ref{rmk:Dmax-euc-cap}) below
and~\cite{Leinstermagnitude2017}, Proposition~4.22).  This connection has
been exploited by Meckes to prove results on magnitude, a closely related
invariant of metric spaces (\cite{MeckesMagnitude2015}, Corollary~7.2).

These infinitely many entropies do not only attain the same maximum
\emph{value} on a given space $X$; better still, there is a single
\emph{probability distribution} that maximizes them all simultaneously.
Passing to the large-scale limit gives a canonical, scale-independent
probability measure on $X$.  For example, if $X$ is isometric to a subset
of Euclidean space then this measure is normalized Lebesgue. It is a general
construction of a `uniform measure' on an abstract metric space.

\paragraph{Measuring diversity}

The backdrop for the theory is a compact Hausdorff topological space \(X\),
equipped with a way to measure the similarity between each pair of
points. This data is encoded as a \emph{similarity kernel}: a continuous
function \(K: X \times X \to [0, \infty)\) taking strictly positive values
  on the diagonal. We call the pair \((X,K)\) a \emph{space with
    similarities}.

In a metric space, we view points as similar if they are close together,
defining a similarity kernel by \(K(x,y) = e^{-d(x,y)}\). Of course,
other choices of kernel are possible, but this particular choice
proves to be a wise one (Example~\ref{eg:metric_1}).  For simplicity, in
this introduction we focus on the case of metric spaces rather than fully
general spaces with similarity.

We would like to quantify the extent to which a probability distribution on
a metric space is spread out across the space, in a way that is sensitive
to distance.  A thinly spread distribution will be said to have `high
diversity', or equivalently `high entropy'.  

\begin{figure}
\setlength{\unitlength}{1mm}
\begin{picture}(120,25)
\cell{18}{5}{b}{\includegraphics[width=36mm]{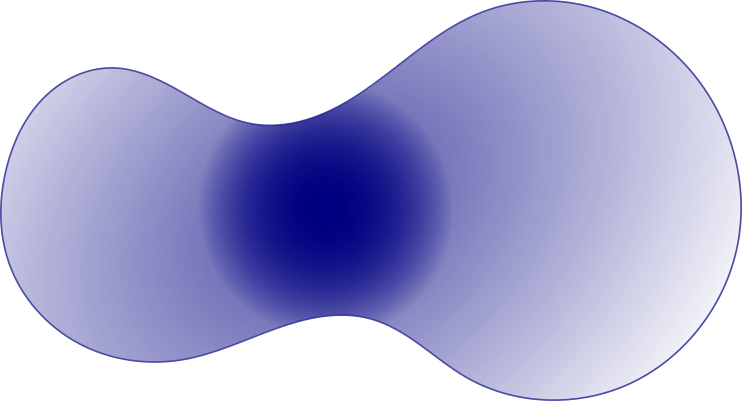}} % one lump
\cell{60}{5}{b}{\includegraphics[width=36mm]{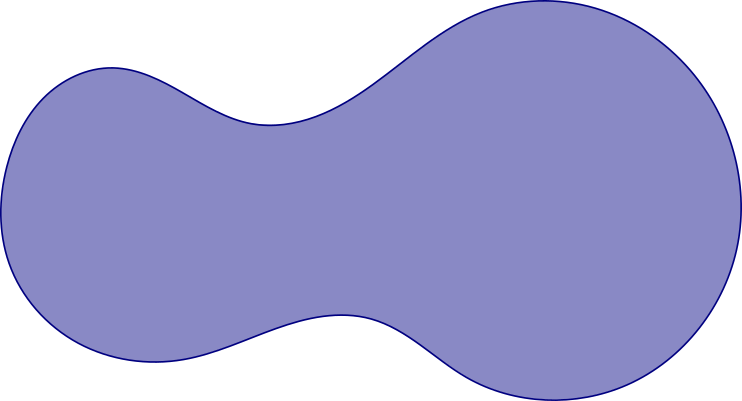}} % uniform
\cell{102}{5}{b}{\includegraphics[width=36mm]{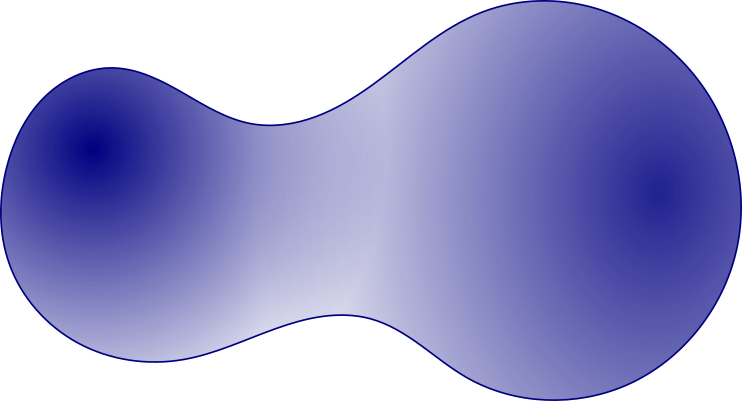}} % two lumps
\cell{18}{0}{b}{(a)}
\cell{60}{0}{b}{(b)}
\cell{102}{0}{b}{(c)}
\end{picture}
\caption{Three probability measures on a subset of the plane. Dark regions
  indicate high concentration of measure.}
\label{fig:distros}
\end{figure}

Figure~\ref{fig:distros} depicts three distributions on the same space.
Distribution~\hard{(a)} is the least diverse, with most of its mass
concentrated in a small region.  Distribution~\hard{(b)} is uniform, and
might therefore seem to be the most diverse or thinly spread distribution
possible.  However, there is an argument that~\hard{(c)} is
more diverse.  In moving from~\hard{(b)} to~\hard{(c)}, some of the mass
has been pushed out to the ends, so a pair of points chosen at random
according to distribution~\hard{(c)} may be more likely to be far apart
than when chosen according to~\hard{(b)}.

One can indeed define diversity in terms of the expected proximity between
a random pair of points.  But that is just one of an infinite family of
ways to quantify diversity, each of which captures something different
about how a distribution is spread across the space.

To define that family of diversity measures, we first introduce the notion
of the \emph{typicality} of a point with respect to a distribution. Given a
compact metric space \(X\), a probability measure \(\mu\) on \(X\), and a
point \(x \in X\), we regard \(x\) as `typical' of \(\mu\) if a point
chosen at random according to \(\mu\) is usually near to \(x\).  Formally,
define a function \(K\mu\) on $X$ by
\[
(K\mu)(x) = \int e^{-d(x,\cdot)} \d\mu.
\]
We call $(K\mu)(x)$ the typicality of $x$, and $1/(K\mu)(x)$ its
atypicality. 

A distribution is widely spread across \(X\) if most points are distant
from most of the mass---that is, if the atypicality function \(1/K\mu\)
takes large values on most of \(X\). A reasonable way to quantify the
diversity of a probability measure \(\mu\), then, is as the average
atypicality of points in \(X\). Here the `average' need not be the
arithmetic mean, but could be a power mean of any order.  Thus, we obtain
an infinite family \((D_q^K)_{q \in [-\infty, \infty]}\) of
diversities. Explicitly, for \(q \neq 1, \pm \infty\), we define the
diversity of order \(q\) of \(\mu\) to be
\[
D_q^K(\mu) 
= 
\left( \int \left( 1/K\mu \right)^{1-q} \d\mu \right)^{1/(1-q)},
\] 
while at \(q=1\) and \(q = \pm \infty\) this expression takes its limiting
values.  The entropy $H_q^K(\mu)$ of order $q$ is $\log D_q^K(\mu)$:
entropy is the logarithm of diversity.

\paragraph{Diversity and entropy}
Any finite set can be given the structure of a compact metric space by
taking all distances between distinct points to be $\infty$.  The
similarity kernel $K = e^{-d(\cdot, \cdot)}$ is then the Kronecker delta
$\delta$.  In this trivial case, the entropy $H_q^\delta$ is precisely the
R\'enyi entropy of order $q$, well-known in information theory.  In
particular, $H_1^\delta$ is Shannon entropy.

Entropy is an important quantitative and conceptual tool in many fields,
including in mathematical ecology, where the exponentials $D_q^\delta$ of
the R\'enyi entropies are known as the Hill numbers and used as measures of
biological diversity~\cite{HillDiversity1973}. In this application, $X$ is
the finite set of species in some ecological community, and $\mu$ encodes
their relative abundances.

However, the Hill numbers fail to reflect a fundamental intuition about
diversity: all else being equal, a biological community is regarded as more
diverse when the species are very different than when they are very
similar.  To repair this deficiency, one can equip the set of species in an
ecological community with a kernel (matrix) \(K\) recording their pairwise
similarities. The choice \(K = \delta\) represents the crude assumption
that each species is completely dissimilar to each other species.  Thus,
for arbitrary $K$, the diversities $D_q^K$ are generalised Hill numbers,
sensitive to species similarity~\cite{LeinsterMeasuring2012}.  Here we
generalise further, from a finite set $X$ to any compact Hausdorff space.

\paragraph{The maximisation theorem}

Crucially, when comparing the diversity of distributions, different values
of the parameter \(q\) lead to different judgements. That is, given a
collection \(M\) of probability measures on a metric space and given
distinct $q, q' \in [0, \infty]$, the diversities \(D_q^K\) and
\(D_{q'}^K\) generally give different orderings to the elements of
\(M\).  Examples in the ecological setting can be found in Section~5
of~\cite{LeinsterMeasuring2012}.

The surprise of our main theorem (Theorem~\ref{thm:main}) is that when it
comes to \emph{maximising} diversity, there is consensus: there is
guaranteed to exist some probability measure \(\mu\) on our space
that maximises \(D_q^K(\mu)\) for every nonnegative \(q\) at
once. Moreover, the diversity of order $q$ of a maximising distribution is
the same for all $q \in [0, \infty]$. Thus, one can speak unambiguously of
the maximum diversity of a compact metric space \(X\)---defined to be
\[
\Dmx{X} = \sup_{\mu} D_q^K(\mu)
\]
for any \(q \in [0,\infty]\)---knowing that there exists a probability
distribution attaining this supremum for all orders $q$.

In the case of a metric space, Theorem~\ref{thm:main} states the
following.

\begin{theorem*}
Let \(X\) be a nonempty compact metric space.
\begin{enumerate}[(i)]
\item
There exists a probability measure \(\mu\) on $X$ that maximises
\(D_q^K(\mu)\) for all \(q \in [0, \infty]\) simultaneously.

\item
The maximum diversity $\sup_{\mu} D_q^K(\mu)$ is independent of
\(q \in [0, \infty]\).
\end{enumerate}
\end{theorem*}

This theorem extends to compact spaces a result that was established for
finite spaces in~\cite{LeinsterMaximizing2016}. (The maximising measure on
a finite metric space is not usually uniform, unless, for instance, the
space is homogeneous.)  While the proof of the result for compact spaces
follows broadly the same strategy as in the finite case, substantial
analytic issues arise.

\paragraph{Geometric connections}

The maximum diversity theorem has geometric significance, linking diversity
measures to fundamental invariants in classical convex geometry and 
geometric measure theory. 

More specifically, Corollary~\ref{cor:comp} of our main theorem connects
maximum diversity with another, more extensively studied invariant of a
metric space: its magnitude.  First introduced as a generalised Euler
characteristic for enriched categories \cite{LeinsterEuler2008,
  Leinstermagnitude2013}, magnitude specialises to metric spaces by way of
Lawvere's observation that metric spaces are enriched
categories~\cite{LawvereMetric1973}. The magnitude \(|X| \in \R\) of a
metric space \(X\) captures a rich variety of classical geometric data,
including some curvature measures of Riemannian manifolds and intrinsic
volumes in \(\ell_1^n\) and Euclidean space. The definition of magnitude
and a few of its basic properties are given in Sections~\ref{S_magnitude}
and~\ref{S_metric} below; \cite{Leinstermagnitude2017} provides a full
survey.

We show that the maximum diversity of a compact space is equal to the
magnitude of a certain subset: the support of any maximising measure
(Sections~\ref{S_prep_lemmas} and~\ref{S_main}). We then use this fact, and
known facts about magnitude, to establish examples of maximum
diversity for metric spaces (Section~\ref{S_metric}).

Many results on magnitude are asymptotic, in the following sense. Given a
space \(X\) with metric \(d\), and a positive real number \(t\), define the
scaled metric space \(tX\) to be the set \(X\) equipped with the metric \(t
\cdot d\). It has proved fruitful to consider, for a fixed metric space
\(X\), the entire family of spaces \((tX)_{t > 0}\) and the
(partially-defined) magnitude function \(t \mapsto |tX|\). For instance, in
\cite{Barcelomagnitudes2018}, Barcel\'o and Carbery showed that the volume
of a compact subset of $\R^n$ can be recovered as the leading term in the
asymptotic expansion of its magnitude function, while in
\cite{Gimperleinmagnitude2017}, Gimperlein and Goffeng showed (subject to
technical conditions) that lower order terms capture surface area and the
integral of mean curvature.

Given this, and given the relationship between magnitude and maximum
diversity, it is natural to consider the function \(t \mapsto
\Dmx{tX}\). Indeed, the asymptotic properties of maximum diversity have
already been shown to be of geometric interest. In
\cite{MeckesMagnitude2015}, Meckes defined the maximum diversity of a
compact metric space to be the maximum value of its diversity of order~2, and
used this definition---now vindicated by our main theorem---to prove the
following relationship between maximum diversity and Minkowski dimension:
\begin{thm}[Meckes~\cite{MeckesMagnitude2015}, Theorem~7.1]
\label{thm:minkowski}
For a compact metric space $X$,
\[
\lim_{t \to \infty} \frac{\log \Dmx{tX}}{\log t} 
= 
\dim_{\mathrm{Mink}}(X),
\]
with the left-hand side defined if and only if the right-hand side is
defined. 
\end{thm}

That is, the Minkowski dimension of $X$ is the growth rate of $\Dmx{tX}$
for large $t$.  Proposition~\ref{prop:preufm-euc} below is a companion
result for the volume of sets $X \subseteq \R^n$:
\[
\lim_{t \to \infty} \frac{\Dmx{tX}}{t^n}
\propto
\textrm{Vol}(X).
\]
Thus, maximum diversity determines dimension and volume.

\paragraph{Entropy and uniform measure}

The maximum diversity theorem implies that every compact metric space $X$
admits a probability measure maximising the entropies of all orders $q$
simultaneously.  Statisticians have long recognised that maximum entropy
distributions are special.  However, the maximum entropy measure on $X$ is
not scale-invariant: if we multiply all distances in $X$ by a constant
factor $t$, the maximising measure changes.

In Section~\ref{S_uniform} we propose a canonical, scale-invariant, choice
of probability measure on a given metric space (subject to conditions), and
call it the \emph{uniform measure}.  It is the limit as $t \to \infty$ of
the maximum entropy measure on $tX$.  We show that in several familiar
cases, this definition captures the intuitive notion of the `obvious'
probability distribution on a space.

\paragraph*{Other notions of entropy} 
There is a vast literature on entropy in geometric contexts.  Here we just
make some brief comments to distinguish entropy in our sense from entropy
in other senses.  

Our entropy is a real invariant of a metric space equipped with a
probability measure.  In contrast, the classical Kolmogorov--Sinai metric
entropy and the related topological entropy of Adler, Bowen, et al.\ are
real invariants of a transformation or flow on a space.  Closer in spirit
is the Kolmogorov $\epsilon$-entropy, which is essentially a simple special
case of our maximum entropy (\cite{LeinsterMaximizing2016},
Section~9). Closer still is differential entropy, which is a real invariant
of a probability density function on a measure space; but unlike our
entropy, it is not defined for an arbitrary probability measure on a metric
space.

\paragraph*{Structure of the paper}

In Section \ref{S_background} we collect various topological and analytic
facts that will be used later. Most of the lemmas in this section
are standard, and the reader may prefer to begin at Section
\ref{S_typicality}.

Sections \ref{S_typicality}, \ref{S_diversity} and \ref{S_magnitude}
introduce our main objects of study---typicality functions, diversity and
entropy, and magnitude---and establish their key properties.  In Section
\ref{S_prep_lemmas} we prove several lemmas and a proposition which form
the scaffolding for the main theorem, proved in Section \ref{S_main}. The
final two sections of the paper specialise from general spaces with
similarities to metric spaces: Section \ref{S_metric} investigates the
relationship between maximum diversity and magnitude, and in Section
\ref{S_uniform} we discuss our definition of the uniform measure on a
compact metric space. A number of open questions are outlined in Section
\ref{S_conjectures}.

\paragraph*{Conventions} 
Throughout, a \demph{measure} on a topological space means a Radon
measure. All measures are positive by default. A function
\(f: \R \to \R\) is \demph{increasing} if \(f(y) \leq f(x)\) for
all \(y \leq x\), and 
% \demph{strictly increasing} if \(f(y) < f(x)\) for all \(y < x\); 
\demph{decreasing} 
% and \demph{strictly decreasing} are used 
similarly.

\paragraph*{Acknowledgements} 
We thank Mark Meckes for many useful conversations, and especially for
allowing us to include Proposition~\ref{prop:euc-unique}, which is due to
him.  Thanks also to Christina Cobbold for helpful suggestions.  TL was
partially supported by a Leverhulme Research Fellowship
(RF-2019-519$\backslash$9).

%%%%%%%%%%%%%%%%% BACKGROUND MATERIAL %%%%%%%%%%%%%%%%%

\section{Topological and analytic preliminaries}
\label{S_background}

\paragraph*{Spaces of functions}

For topological spaces $X$ and $Y$, let \(\cn{Top}(X,Y)\) denote the set of
continuous maps from $X$ to $Y$.

When \(X\) is compact and \(Y\) is a metric space, the compact-open
topology and the topology of uniform convergence on \(\cn{Top}(X,Y)\)
coincide.  (This follows, for example, from Theorems~46.7 and~46.8 in
\cite{MunkresTopology2000}.)  We will be exclusively concerned with cases
where $X$ is compact and $Y$ is metric, and we will always understand
$\cn{Top}(X, Y)$ to be equipped with this topology. In particular, $C(X) =
\cn{Top}(X, \R)$ has the topology induced by the uniform norm
$\|\cdot\|_\infty$.

%\begin{lem}\label{lem:c.o.uniform}
%Let \(X\) be a compact topological space and \(Y\) a metric space. Then the
%compact-open topology and the topology of uniform convergence on
%\(\cn{Top}(X, Y)\) are equal.
%\end{lem}
%
%\begin{proof}
%This follows from Theorems~46.7 and~46.8 in \cite{MunkresTopology2000}.
%\end{proof}
%
%We will be exclusively concerned with cases where $X$ is compact and $Y$ is
%metric, and we will always understand $\cn{Top}(X, Y)$ to be equipped with
%the unique topology of Lemma~\ref{lem:c.o.uniform}.  In particular, $C(X) =
%\cn{Top}(X, \R)$ has the topology induced by the uniform norm
%$\|\cdot\|_\infty$.

\begin{lem}\label{lem:uniform_expbl}
Let \(X\) be any topological space, \(Y\) a compact Hausdorff space, and
\(Z\) a metric space.  A map \(f: X \times Y \to Z\) is continuous if and
only if the map \(\overline{f}: X \to \cn{Top}(Y,Z)\) given by
\(\overline{f}(x)(y) = f(x,y)\) is continuous.
\end{lem}

\begin{proof}
This follows from the standard properties of the compact-open topology
(\cite{BorceuxHandbook1994a}, Proposition~7.1.5).
\end{proof}

We will make repeated use of the following elementary fact.

\begin{lem}\label{lem:uniform_functorial}
Let \(X\) be a compact topological space, \(Y\) and \(Y'\) metric spaces,
and \(\phi: Y \to Y'\) a continuous function. Then the induced map
\[
\phi \circ - : \cn{Top}(X, Y) \to \cn{Top}(X, Y')
\]
is continuous.
\end{lem}

%\begin{proof}
%It is elementary that $\phi \circ -$ is continuous with respect to the
%compact-open topology on both the domain and the codomain. 
%\end{proof}

\paragraph*{Spaces of measures}

\glob{From now until Definition \ref{def_powermean}, let $X$ denote a
compact Hausdorff space.}  Equip the vector space $C(X)$ with the norm
$\|\cdot\|_\infty$.  The Riesz representation theorem identifies its
topological dual $C(X)^*$ with the space $M(X)$ of finite signed measures
on \(X\).  The dual norm on $M(X)$ is the total variation norm, $\|\mu\| =
|\mu|(X)$, and the dual pairing is
\begin{equation}
\label{eq:pairing}
\begin{array}{cccc}
\langle -, - \rangle:   &C(X) \times M(X)       &\to    &\R,            \\
                        &(f, \mu)               &\mapsto&\int_X f \d\mu.
\end{array}
\end{equation}
We will always understand $M(X)$ and its subsets to be equipped with the
weak* topology.  Denote by $P(X)$ the set of probability measures on $X$,
and by $P_\leq(X)$ the set of measures $\mu$ such that $\mu(X) \leq 1$.  By
the Banach--Alaoglu theorem, $P(X)$ and $P_\leq(X)$ are compact Hausdorff.

The pairing map~\eqref{eq:pairing} is not in general continuous.  However:

\begin{lem}\label{lem:pairing_cts}
Let \(Q\) be a closed bounded subset of \(M(X)\). Then:
\begin{enumerate}[(i)]
\item 
\label{part:pairing_cts_1}
the assignment \(f \mapsto \langle
f, - \rangle\) defines a continuous map \(C(X) \to C(Q)\);

\item 
\label{part:pairing_cts_2}
the restricted pairing map \(\langle -, - \rangle: C(X) \times Q
\to \R\) is continuous.
\end{enumerate}
\end{lem}

\begin{proof}
For (\ref{part:pairing_cts_1}), first note that for each \(f \in C(X)\),
the map \(\langle f, - \rangle : Q \to \R\) is continuous, by definition of
the weak* topology. To show that the resulting map \(C(X) \to C(Q)\) is
continuous, let \(f,g \in C(X)\). Then
\[
\|\langle f, - \rangle - \langle g, - \rangle \|_\infty 
= 
\sup_{\mu \in Q} |\langle f - g, \mu \rangle|
\leq 
\|f-g\|_\infty \sup_{\mu \in Q} \|\mu\|,
\] 
and \(\sup_{\mu \in Q} \|\mu\|\) is finite as \(Q\) is bounded.

Part (\ref{part:pairing_cts_2}) follows from Lemma~\ref{lem:uniform_expbl},
since \(Q\) is compact (by the Banach--Alaoglu theorem) and Hausdorff.
\end{proof}

\paragraph*{Supports}

The \demph{support} of a function \(f: X \to [0, \infty)\) is $\supp f =
f^{-1}(0, \infty)$.  Note that we use this set rather than its closure.

Every measure $\mu$ on $X$ has a \demph{support} $\supp \mu$, which is the
smallest closed set satisfying \(\mu(X \setminus \supp \mu) = 0\).  (Recall
our convention that `measure' means `positive Radon measure', and see, for
instance, Chapter~III, \S2, No.~2 of~\cite{BourbakiIntegration1965}.)  It
is characterised by
\[
\supp \mu 
= 
\{x \in X \st 
\mu(U) > 0 \text{ for all open neighbourhoods \(U\) of \(x\)}\},
\] 
and has the property that $\int_X f \d\mu = \int_{\supp \mu} f \d\mu $ for
all $f \in L^1(X, \mu)$.

\begin{lem}
\label{lem:disjoint_supps}
Let \(\mu\) be a measure on \(X\),and let \(f: X \to [0,\infty)\) be a
continuous function. Then
$
\supp f \cap \supp \mu \neq \emptyset 
\iff 
\int_X f \d\mu > 0.
$
\end{lem}

\begin{proof}
The forwards implication is Proposition~9 in Chapter~III, \S2, No.~3
of~\cite{BourbakiIntegration1965}, and the backwards implication is trivial.
\end{proof}

\paragraph*{Approximations to the identity}

Later, we will want to approximate Dirac measures $\delta_x$ by probability
measures that are absolutely continuous with respect to some fixed measure
$\mu$.  We will use:

\begin{lem}
\label{lem:approximate_delta}
Let \(\mu\) be a measure on \(X\) and \(x \in \supp \mu\). For each
equicontinuous set of functions \(E \subseteq C(\supp \mu)\) and each
\(\epsilon > 0\), there exists a nonnegative function \(u \in C(X)\) such
that \(u\mu\) is a probability measure and for all \(f \in E\),
\[
\left| \int_X f \d(u\mu) - f(x) \right| \leq \epsilon.
\]
\end{lem}

\begin{proof}
By equicontinuity, we can choose a subset \(U \subseteq \supp \mu\),
containing \(x\) and open in \(\supp \mu\), such that \(|f(y) - f(x)| \leq
\epsilon\) for all \(y \in U\) and \(f \in E\).

By Urysohn's lemma, we can choose a nonnegative function \(u \in C(\supp
\mu)\) such that \(\supp u \subseteq U\) and \(u(x) > 0\). Then
\(\int_{\supp \mu} u \d\mu > 0\), so by rescaling we can arrange that
\(\int_{\supp\mu} u \d\mu = 1\).

By Tietze's theorem, \(u\) can be extended to a nonnegative
continuous function on \(X\), and then \(u\mu\) is a probability measure on
\(X\). Moreover, for all \(f \in E\),
\[
\left| \int_X f \d(u \mu) - f(x) \right| 
= 
\left| \int_U \bigl(f(y) - f(x)\bigr) u(y) \d\mu(y) \right|       
\leq 
\epsilon \int_U u(y) \d\mu(y) 
= 
\epsilon,
\]
as required.
\end{proof}

We will also want to approximate any probability measure on $\R^n$
by measures that are absolutely continuous with respect to Lebesgue measure
$\lambda$.  We do this in the following standard way.  Let $G \in
L^1(\R^n)$ with $\int G = 1$.  Define functions $(G_t)_{t > 0}$ on $\R^n$
by $G_t(x) = t^n G(tx)$.  Then $G_t \in L^1(\R^n)$ and $\int G_t = 1$ for
every $t$.  The convolution $G_t * \mu$ with any finite signed measure
$\mu$ on $\R^n$ also belongs to $L^1(\R^n)$ (Proposition~8.49
of~\cite{FollandReal1999}).

\begin{lem}
\label{lem:approx-conv}
Let $G \in L^1(\R^n)$ with $\int_{\R^n} G \d\lambda = 1$, and let $f \in
C(\R^n)$ be a function of bounded support.  Then for all probability
measures $\mu$ on $\R^n$,
\[
\int_{\R^n} f \cdot (G_t * \mu) \d\lambda 
\to
\int_{\R^n} f \d\mu
\quad
\text{ as } t \to \infty,
\]
uniformly in $\mu$.
\end{lem}

\begin{proof}
Define $\tilde{G} \in L^1(\R^n)$ by $\tilde{G}(x) = G(-x)$.  It is
elementary that
\[
\int_{\R^n} f \cdot (G_t * \mu) \d\lambda 
-
\int_{\R^n} f \d\mu
=
\int_{\R^n} \bigl( f * \tilde{G}_t - f \bigr) \d\mu
\]
for all finite signed measures $\mu$ on $\R^n$.  Hence when $\mu$ is a
probability measure,
\[
\biggl|
\int_{\R^n} f \cdot (G_t * \mu) \d\lambda 
-
\int_{\R^n} f \d\mu
\biggr|
\leq
\bigl\| f * \tilde{G}_t - f \bigr\|_\infty 
\to 0
\]
as $t \to \infty$, by Theorem~8.14(b) of~\cite{FollandReal1999}.
\end{proof}

\paragraph*{Integral power means}

Here we review the theory of power means of a
real-valued function on an arbitrary probability space $(X, \mu)$.
% (suspending the convention that $X$ denotes a compact Hausdorff
% topological space).

A function \(f: X \to [0,\infty)\) is \demph{essentially bounded} if
  \(\esssup_\mu(f)\) is finite.

\begin{defn}\label{def_powermean}
Let \((X, \mu)\) be a probability space and let $f: X \to [0, \infty)$ be a
measurable function such that both $f$ and $1/f$ are essentially
bounded.  We define for each $t \in [-\infty, \infty]$ a real number 
$
M_t(\mu, f) \in (0, \infty)$,
the \demph{power mean of \(f\) of order \(t\), weighted by \(\mu\)}, by
\begin{align}
\label{powermean_def_1}
M_t(\mu,f) = \left(\int_X f^t \d\mu \right)^{1/t}
\end{align}
when $t \in (-\infty, 0) \cup (0, \infty)$, and in the remaining cases by 
\begin{align*}
M_0(\mu,f)      &
= 
\exp\left(\int_X \log f\d\mu\right), \\
M_{\infty} (\mu, f)    & 
= 
\esssup_\mu f,  \\
M_{-\infty} (\mu,f)     &
= 
\essinf_\mu f.
\end{align*}
\end{defn}

In the case of a finite set $X = \{1, \ldots, n\}$, the mean of order $0$
is the classical weighted geometric mean $\prod_{i = 1}^n f(i)^{\mu\{i\}}$.

\begin{rmk}\label{rmk:powermean_dual}
We have made the assumption that $f$ and $1/f$ are essentially bounded, or
equivalently that $\essinf_\mu(f) > 0$ and $\esssup_\mu(f) < \infty$.  This
guarantees that $f^t \in L^1(X, \mu)$ for all $t \in (-\infty, \infty)$ and
that $M_t(\mu, f) \in (0, \infty)$ for all $t \in [-\infty, \infty]$.  If
$f$ satisfies our assumption then so does $1/f$, and we have a duality
formula: 
\[
M_{-t}(\mu,f) 
= 
\frac{1}{M_t(\mu, 1/f)}.
\]
\end{rmk}

\begin{prop}\label{prop:powermean_mono}
Let \((X, \mu)\) be a probability space and let $f: X \to [0, \infty)$ be a
measurable function such that both $f$ and $1/f$ are essentially
bounded.  
\begin{enumerate}[(i)]
\item 
\label{part:pm-const}
If there is some constant $c$ such that $f(x) = c$ for almost all $x \in
X$, then $M_t(\mu, f) = c$ for all $t \in [-\infty, \infty]$.

\item
\label{part:pm-strict}
Otherwise, $M_t(\mu, f)$ is strictly increasing in $t \in [-\infty,
  \infty]$. 
\end{enumerate}
\end{prop}

\begin{proof}
Part~(\ref{part:pm-const}) is trivial.  Part~(\ref{part:pm-strict}) is
proved in Section~6.11 of~\cite{HardyInequalities1952} in the case where
$X$ is a real interval and $\mu$ is determined by a density function, and
the proof extends without substantial change to an arbitrary probability
space.
\end{proof}

\begin{prop}\label{prop:powermean_cts}
Let \((X, \mu)\) be a probability space and let $f: X \to [0, \infty)$ be a
measurable function such that both $f$ and $1/f$ are essentially
bounded.  Then $M_t(\mu, f)$ is continuous in $t \in [-\infty, \infty]$.
\end{prop}

\begin{proof}
Again, this is proved in the case of a real interval in Section~6.11
of~\cite{HardyInequalities1952}.  The generalisation to an arbitrary
probability space is sketched as Exercise~1.8.1
of~\cite{NiculescuConvex2006}, although the hypotheses on $f$ there are
weaker than ours, and at $t = 0$ only continuity from the right is proved.  
Under our hypotheses on $f$, continuity from the left then follows from the
duality of Remark~\ref{rmk:powermean_dual}.
\end{proof}

\paragraph*{Differentiation under the integral sign}

We will need the following standard result (Theorem~6.28
of~\cite{KlenkeProbability2013}).

\begin{lem}\label{lem:diff_int}
Let \((X,\mu)\) be a measure space and \(J \subseteq \R\)
an open interval. Let \(f: X \times J \to \R\) be a
map with the following properties:
\begin{enumerate}[(i)]
\item 
\label{part:di-int}
for all $t \in J$, the map $f(-, t): X \to \R$ is integrable;

\item 
\label{part:di-diff}
for almost all \(x \in X\), the map \(f(x, -): J \to \R\) is
differentiable; 

\item 
\label{part:di-bound}
there is an integrable function $h: X \to \R$ such that for all \(t \in
  J\), for almost all $x \in X$, we have \(\bigl|\tfrac{\partial f}{\partial
    t} (x, t)\bigr| \leq h(x)\).
\end{enumerate}
Then $\tfrac{\partial f}{\partial t} (-, t): X \to \R$ is integrable for
each $t \in J$, and the function $t \mapsto \int_X f(-, t) \d\mu$ on $J$ is
differentiable with derivative $t \mapsto \int_X \tfrac{\partial f}{\partial
  t} (-, t) \d\mu$.   
\end{lem}

%%%%%%%%%%%%%%%%% TYPICALITY FUNCTIONS %%%%%%%%%%%%%%%%%

\section{Typicality}
\label{S_typicality}

The setting for the rest of this paper is a space $X$ equipped with a
notion of similarity or proximity between points in $X$ (which may or may
not be derived from a metric).  In this section, we show how any probability
measure on $X$ gives rise to a `typicality function' on $X$, whose value at
a point $x$ indicates how concentrated the measure is near $x$.

\begin{defn}\label{def_similarity}
Let \(X\) be a compact Hausdorff space. A \demph{similarity kernel} on
\(X\) is a continuous function \(K: X \times X \to [0, \infty)\) satisfying
  \(K(x,x)>0\) for all \(x \in X\). The pair \((X,K)\) is a
  \demph{(compact Hausdorff) space with similarities}.
\end{defn}

Since we will only be interested in compact Hausdorff spaces, we omit the
words `compact Hausdorff' and simply refer to spaces with similarities.

\begin{example}\label{eg:ecosystem_1}
In ecology, there has been vigorous discussion of how best to quantify the
diversity of a biological community.  This is a conceptual and mathematical
challenge, quite separate from the practical and statistical difficulties,
and many dozens of diversity measures have been proposed over~70 years of
debate in the ecological literature~\cite{MagurranMeasuring2004}.

Any realistic diversity measure should reflect the degree of variation
between the species present.  All else being equal, a lake
containing four species of carp should be counted as less diverse than a lake
containing four very different species of fish.  The similarity between
species may be measured genetically, phylogenetically, functionally, or in
some other way (as discussed in~\cite{LeinsterMeasuring2012}); how it
is done will not concern us here.

Mathematically, we take a finite set $X = \{1, \ldots, n\}$ (whose elements
represent the species) and a real number $Z_{ij} \geq 0$ for each pair $(i,
j)$ (representing the degree of similarity between species $i$ and $j$).  A
similarity coefficient $Z_{ij} = 0$ means that species $i$ and $j$ are
completely dissimilar, and we therefore assume that $Z_{ii} > 0$ for all
$i$.  Thus, $Z = (Z_{ij})$ is an $n \times n$ nonnegative real matrix with
strictly positive entries on the diagonal.  In the language of
Definition~\ref{def_similarity}, this is the case of finite spaces with
similarities: $X$ has the discrete topology, and the similarity kernel $K$
is given by $K(i, j) = Z_{ij}$.  When $Z$ is the identity matrix, $K$ is
the Kronecker delta.

Many ways of assigning inter-species similarities are calibrated on a
scale of $0$ to $1$, with $Z_{ii} = 1$ for all $i$ (each species is
identical to itself).  For example, percentage genetic similarity gives
similarity coefficients in $[0, 1]$, as does the similarity measure
$e^{-d(i, j)}$ derived from a metric $d$ and discussed below.  The simplest
possible choice of $Z$ is the identity matrix, embodying the crude
assumption that different species have nothing in common whatsoever.
\end{example}

\begin{example}\label{eg:metric_1}
Any compact metric space $(X, d)$ can be regarded as a space with
similarities $(X, K)$ by putting
\[
K(x, y) = e^{-d(x, y)}
\]
($x, y \in X$).  The extreme case where $d(x, y) = \infty$ for all $x
\neq y$ produces the Kronecker delta. 

Although the negative exponential is not the only reasonable function
transforming distances into similarities, it turns out to be a particularly
fruitful choice.  It is associated with the very fertile theory of the
magnitude of metric spaces (surveyed in~\cite{Leinstermagnitude2017}).
Moreover, the general categorical framework of magnitude all but
forces this choice of transformation, as explained in Example~2.4(3)
of~\cite{Leinstermagnitude2017}.
\end{example}

In the examples above, the similarity kernel is \demph{symmetric}: \(K(x,y)
= K(y,x)\) for all \(x,y \in X\). In such cases we say \((X,K)\) is a
\demph{symmetric space with similarities}. We do not include symmetry in
the definition of similarity kernel, partly because asymmetric similarity
matrices occasionally arise in ecology, and also because of the argument of
Gromov (\cite{GromovMetric2001}, p.~xv) and Lawvere
(\cite{LawvereMetric1973}, p.~138--9) that the symmetry condition in the
definition of metric can be too restrictive.  To obtain our main result,
however, it will be necessary to assume symmetry.

Most measures of biological diversity depend (at least in part) on the
relative abundance distribution $\vec{p} = (p_1, \ldots, p_n)$ of the
species, where `relative' means that the $p_i$ are normalised to sum to
$1$.  Multiplying the similarity matrix $Z$ by the column vector $\vec{p}$
gives another vector $Z\vec{p}$, with $i$th entry
\[
(Z\vec{p})_i = \sum_j Z_{ij} p_j.
\]
This is the expected similarity between an individual of species $i$ and an
individual chosen at random.  Thus, $(Z\vec{p})_i$ measures how typical
individuals of species $i$ are within the community.  The generalisation to
an arbitrary space with similarities is as follows.

\begin{defn}\label{def_typicality}
Let $(X, K)$ be a space with similarities.  For each \(\mu \in M(X)\) and
\(x \in X\), define
\[
(K\mu)(x) = \int_X K(x,-) \d\mu \in \R.
\]
This defines a function \(K\mu: X \to \R\), the \demph{typicality
  function} of \((X,K,\mu)\).
\end{defn}

When $\mu$ is a probability measure (the case of principal interest),
$(K\mu)(x)$ is the expected similarity between $x$ and a random point.  It
therefore detects the extent to which $x$ is similar, or near, to sets of
large measure.

In the next section, we will define entropy and diversity in terms of the
typicality function $K\mu$.  For that, we will need to know that $K\mu$
satisfies some analytic conditions, which we now establish.

\glob{For the rest of this section, let $(X, K)$ be a space with
  similarities.}  The next lemma follows directly from
  Lemma~\ref{lem:uniform_expbl}. 

\begin{lem}\label{lem:Kbar_cts}
The function $\overline{K}: X \to C(X)$ defined by $x \mapsto K(x, -)$ is
continuous. 
\end{lem}

%\begin{proof}
%This follows from Lemma~\ref{lem:uniform_expbl}.
%\end{proof}

\begin{lem}\label{lem:Kmu_cts}
For each \(\mu \in M(X)\), the function \(K\mu: X \to \R\) is continuous.
\end{lem}

\begin{proof}
Note that \(K\mu\) is the composite
\[
X \xto{\overline{K}} C(X) \xto{\langle-,\mu\rangle} \R.
\]
By Lemma \ref{lem:Kbar_cts}, $\overline{K}$ is continuous, and
\(\langle-,\mu\rangle = \int_X - \d\mu\) is a continuous linear functional.
Hence $K\mu$ is continuous.
\end{proof}

\begin{lem}\label{lem:K_*_cts}
The map
\[
\begin{array}{cccc}
K_*:    &P(X)        &\to       &C(X) \\
        &\mu         &\mapsto   &K\mu
\end{array}
\]
is continuous.
\end{lem}

\begin{proof}
Both \(X\) and \(P(X)\) are compact Hausdorff so, applying
Lemma~\ref{lem:uniform_expbl} twice, an equivalent statement is that
the map
\[
\begin{array}{ccc}
X       &\to            &C(P(X))        \\
x       &\mapsto        &(K-)(x) = (\mu \mapsto (K\mu)(x))
\end{array}
\]
is continuous. This map is the composite
\[
X \xto{\overline{K}} C(X) \to C(P(X)),
\]
where the second map is \(f \mapsto \langle f,-\rangle\) and is continuous
by Lemma~\ref{lem:pairing_cts}(\ref{part:pairing_cts_1}). Hence, \(K_*:
P(X) \to C(X)\) is continuous.
\end{proof}

\begin{prop}\label{prop:Kmu_props}
For every measure $\mu$ on $X$, 
\begin{enumerate}[(i)]%[label=(\roman*)]
\item 
\label{part:Kp-supp}
\(\supp K\mu \supseteq \supp \mu\);

\item 
\label{part:Kp-eb}
both \(K\mu\) and \(1/K\mu\) are essentially bounded with respect to
\(\mu\). 
\end{enumerate}
\end{prop}

\begin{proof}
For~(\ref{part:Kp-supp}), let \(x \in \supp \mu\). Since \(K\) is positive
on the diagonal, \(x \in \supp K(x,-)\), so \(\supp \mu \cap \supp K(x,-)
\neq \emptyset\). Hence by Lemma~\ref{lem:disjoint_supps},
\[
(K\mu)(x) = \int_X K(x,-) \d\mu > 0.
\]
For~(\ref{part:Kp-eb}), $\supp\mu$ is compact and \(K\mu\) is continuous with
\(K\mu\big|_{\supp \mu} > 0\).  So both $K\mu$ and $1/K\mu$ are
bounded on $\supp\mu$, hence essentially bounded on $X$.
\end{proof}

%%%%%%%%%%%%%%%%% DIVERSITY  %%%%%%%%%%%%%%%%%

\section{Diversity and entropy}
\label{S_diversity}

Here we introduce the main subject of the paper: a one-parameter family of
functions that quantify the degree of spread of a probability measure on
a compact Hausdorff space \(X\), with respect to a chosen notion of
similarity between points of \(X\).  

Take a probability measure $\mu$ on a space with similarities $(X, K)$.
Intuitively, $\mu$ is widely spread across $X$ if most points are
dissimilar to most of the rest of $X$, interpreting `most' in terms of
$\mu$.  Equivalent ways to say this are that the typicality function $K\mu:
X \to \R$ takes small values on most of $X$, or that the `atypicality'
function $1/K\mu$ takes large values on most of $X$.  Ecologically, a
community is diverse if it is predominantly made up of species that are
unusual or atypical within that community (for example, many rare and
highly dissimilar species).

Diversity of $\mu$ is, therefore, defined as mean atypicality.  It is
useful to consider not just the arithmetic mean, but the power means of all
orders:

\begin{defn}\label{def_diversity}
Let \((X, K)\) be a space with similarities and let $q \in [-\infty,
  \infty]$.  The \demph{diversity of order \(q\)} of a probability measure
$\mu$ on $X$ is
\[
D_q^K(\mu) = M_{1-q}(\mu, 1/K\mu) \in (0, \infty).
\]
The \demph{entropy of order $q$} of $\mu$ is $H_q^K(\mu) = \log D_q^K(\mu)$.
\end{defn}

By the duality of Remark~\ref{rmk:powermean_dual}, an equivalent definition is
\[
D_q^K(\mu) = 1/M_{q - 1}(\mu, K\mu).
\]
On the right-hand side, the denominator is the mean typicality of a point
in $X$, which is a measure of \emph{lack} of diversity; its reciprocal is
then a measure of diversity.  The power means in this formula and
Definition~\ref{def_diversity} are well-defined because $K\mu$ and $1/K\mu$
are essentially bounded with respect to \(\mu\)
(Proposition~\ref{prop:Kmu_props}).

Explicitly, 
\begin{align*}\label{4}
D_q^K(\mu)= 
\begin{cases}
\left( \int_{X} \left( K\mu \right)^{q-1} \d\mu \right)^{1/(1-q)}       & 
\text{ if } q \in (-\infty, 1) \cup (1, \infty),\\
\exp \left( -\int_X \log (K\mu) \d\mu\right)    & 
\text{ if } q = 1,              \\
1 / \esssup_\mu K\mu        & 
\text{ if } q = \infty,         \\
1 / \essinf_\mu K\mu        & 
\text{ if } q = - \infty.
\end{cases}
\end{align*}
We usually work with the diversities $D_q^K$ rather than the entropies
$H_q^K$, but evidently it is trivial to translate results on diversity into
results on entropy.

\begin{example}
Let $X$ be the set $\{1, \ldots, n\}$ with the discrete topology, let $K$
be the Kronecker delta on $X$, and let $\mu$ be the uniform measure on $X$.
Then $K\mu \equiv 1/n$, so $D_q^K(\mu) = n$ and $H_q^K(\mu) = \log n$ for
all $q$.  This conforms to the intuition that the larger we take $n$ to be,
the more thinly spread the uniform measure on $\{1, \ldots, n\}$ becomes.
\end{example}

The next two examples also concern the finite case.  They are described in
terms of the ecological scenario of Example~\ref{eg:ecosystem_1}.  Thus, $X
= \{1, \ldots, n\}$ is a set of species, $Z_{ij} = K(i, j)$ is the
similarity between species $i$ and $j$, and $\mu = \vec{p} = (p_1, \ldots,
p_n)$ gives the proportions in which the species are present.

\begin{example}
\label{eg:fin-div-1}
Put $Z = I$ (distinct species have nothing in common).  Then 
\[
D_0^I(\vec{p}) 
=
\sum_{i \in \supp\vec{p}} p_i \cdot \frac{1}{p_i}
=
|\supp\vec{p}|.
\]
This is just the number of species present.  It is the simplest diversity
measure of all.  But it takes no account of species abundances beyond
presence and absence, whereas, for instance, a community of two species is
ordinarily considered more diverse if they are equally abundant than if
their proportions are $(0.99, 0.01)$.

The diversity of order $1$ is
\[
D_1^I(\vec{p}) 
= 
\exp\Biggl( - \sum_{i \in \supp\vec{p}} p_i \log p_i\Biggr)
=
\prod_{i \in \supp\vec{p}} p_i^{-p_i}
\]
and the entropy $H_1^I(\vec{p}) = \log D_1^I(\vec{p})$ of order $1$ is the
Shannon entropy $-\sum p_i \log p_i$.  The diversity of order $2$
is
\[
D_2^I(\vec{p}) 
=
1\biggl/\sum_{i = 1}^n p_i^2.
\]
The denominator is the probability that two individuals chosen at
random are of the same species, and $D_2^I(\vec{p})$ itself is the expected
number of such trials needed in order to obtain a matching pair.  The
diversity of order $\infty$ is
\[
D_\infty^I(\vec{p}) 
=
1\bigl/\max_i p_i,
\]
which measures the extent to which the community is dominated by a single
species.  All four of these diversity measures (or simple transformations
of them) are used by ecologists~\cite{MagurranMeasuring2004}.
For a general parameter value $q \neq 1, \pm\infty$, the diversity of order
$q$ is 
\[
D_q^I(\vec{p}) 
= 
\Biggl( \sum_{i \in \supp\vec{p}} p_i^q \Biggr)^{1/(1 - q)}.
\]
In ecology, $D_q^I$ is known as the \demph{Hill number} of order
$q$~\cite{HillDiversity1973}, and in information theory, $H_q^I = \log
D_q^I$ is called the \demph{R\'enyi entropy} of order
$q$~\cite{Renyimeasures1961}.  For reasons explained in
Remark~\ref{rmk:q-range}, we usually restrict to $q \geq 0$.

The parameter $q$ controls the emphasis placed on rare or common species.
Low values of $q$ emphasise rare species; high values emphasise common
species.  At one extreme, $D_0^I$ depends only on presence/absence, thus
attaching as much importance to rare species as common ones.  At the other,
$D_\infty^I$ depends only on the abundance of the most common species,
completely ignoring rarer ones.

If a community loses one or more rare species, while at the same time the
remaining species become more evenly balanced, its low-order diversity
falls but its high-order diversity rises.  For example, $D_q^I$
measures the relative abundance distribution $(0.5, 0.5, 0)$ as less
diverse than $(0.8, 0.1, 0.1)$ when $q < 0.853$, but more diverse for
all higher values of $q$.

% Gnuplot code:
% plot [0.852:0.854] (0.8**x + 2*(0.1**x))**(1/(1 - x)), 2
% Crossing point is ~0.8526

The moral is that when judging which of two communities is the more
diverse, the answer depends critically on the parameter~$q$.  Different
values of $q$ may produce opposite judgements.
\end{example}

\begin{example}
\label{eg:fin-div-2}
Still in the ecological setting, consider now a general similarity matrix
$Z$ encoding the similarities between species (as
in Example~\ref{eg:ecosystem_1}).  The diversity measures $D_q^Z$ and the
role of the parameter $q$ can be understood much as in the case $Z = I$,
but now incorporating inter-species similarity.  For instance,
\[
D_2^Z(\vec{p}) = 1\Bigl/\sum_{i, j} p_i Z_{ij} p_j
\]
is the reciprocal expected similarity between a random pair of 
individuals (rather than the reciprocal probability that they are of the
same species), and
\[
D_\infty^Z(\vec{p}) = 1\Bigl/ \max_{i \in \supp\vec{p}} (Z\vec{p})_i
\]
reflects the dominance of the largest cluster of species (rather
than the largest single species).
\end{example}

\begin{example}
\label{eg:div2}
Let $(X, K)$ be an arbitrary space with similarities.  Among all the
diversity measures $(D_q^K)_{q \in [0, \infty]}$, one with
especially convenient mathematical properties is the diversity of order
$2$:
\[
D_2^K(\mu) 
= 
\frac{1}{\int_X \int_X K(x, y) \d\mu(x) \d\mu(y)}.
\]
For instance, Meckes used $D_2^K$, and more particularly the maximum
diversity $\sup_{\mu \in P(X)} D_2^K(\mu)$ of order~$2$, to prove results
on the Minkowski dimension of metric spaces (\cite{MeckesMagnitude2015},
Section~7).
\end{example}

We now establish the basic analytic properties of diversity.  First, we
show that when $\mu$ is fixed, \(D_q^K(\mu)\) is a continuous and
decreasing function of its order \(q\).  Second, we prove the more
difficult result that when \(q \in (0,\infty)\) is fixed, 
$D_q^K(\mu)$ is continuous in the measure $\mu$.

\begin{prop}\label{prop:diversity_cts_q}
Let \((X, K)\) be a space with similarities and let $\mu \in P(X)$.
\begin{enumerate}[(i)]
\item 
\label{part:dcq-cts}
\(D_q^K(\mu)\) is continuous in its order \(q \in [-\infty, \infty]\).

\item 
\label{part:dcq-dec}
If \(K\mu\) is constant on the support of \(\mu\), then the function \(q
\mapsto D_q^K(\mu)\) is constant on \([-\infty, \infty]\); otherwise, it is
strictly decreasing in \(q \in [-\infty, \infty]\).
\end{enumerate}
\end{prop}

\begin{proof}
The two parts follow from 
Propositions~\ref{prop:powermean_cts} and~\ref{prop:powermean_mono},
respectively. 
\end{proof}

\begin{rmk}\label{rmk:typicality_constant}
A central role will be played by measures $\mu$ satisfying the first case
of Proposition~\ref{prop:diversity_cts_q}(\ref{part:dcq-dec}).  We call
$\mu$ \demph{balanced} if the function $K\mu$ is constant on $\supp\mu$.
(In~\cite{LeinsterMaximizing2016}, for finite $X$, such measures were
called `invariant'.)  Equivalently, $\mu$ is balanced if $D_q^K(\mu)$ is
constant over $q \in [-\infty, \infty]$.  If $(K\mu)|_{\supp\mu}$ has
constant value $c$ then $D_q^K(\mu)$ has constant value $1/c$.
\end{rmk}

\begin{prop}\label{prop:diversity_cts_mu}
Let \((X, K)\) be a space with similarities.  For every \(q \in (0,
\infty)\), the diversity function \(D_q^K: P(X) \to \R\) is continuous.
\end{prop}

(Recall that we always use the weak$^*$ topology on $P(X)$.)

The proof of Proposition~\ref{prop:diversity_cts_mu} takes the form of
three lemmas, addressing the three cases \(q \in (1,\infty)\), \(q \in
(0,1)\) and \(q = 1\).  

\begin{lem}
For every \(q \in (1,\infty)\), the diversity function \(D_q^K: P(X)
\to \R\) is continuous.
\end{lem}

\begin{proof}
The map \(\mu \mapsto 1/D_q^K(\mu)\) is the composite
\[
P(X) 
\xto{\triangle}
P(X) \times P(X) 
\xto{K_* \times \text{Id}} 
C(X) \times P(X) 
\xto{(-)^{q - 1} \times \text{Id}} 
C(X) \times P(X) 
\xto{\langle -, - \rangle} 
\R 
\xto{(-)^{1/(q - 1)}}
\R.
\]
Here $\triangle$ is the diagonal, which is certainly continuous.  The map
$K_*$ was defined and proved to be continuous in Lemma~\ref{lem:K_*_cts},
and $(-)^{q - 1}: C(X) \to C(X)$ is continuous by
Lemma~\ref{lem:uniform_functorial}.  The restricted pairing $\langle -, -
\rangle$ on $C(X) \times P(X)$ is continuous by
Lemma~\ref{lem:pairing_cts}.  Finally, \((-)^{1/(q-1)}\) is evidently
continuous.  Hence $1/D_q^K$ is continuous.  But $D_q^K$ takes values in
$(0, \infty)$, so is itself continuous.
\end{proof}

The case $q \in (0, 1)$ is harder.  In the following proof, most of the
work is caused by the possibility that \((K\mu)(x) = 0\) for some \(x\),
in which case the function \((K\mu)^{q-1}\) is not defined everywhere.

\begin{lem}\label{lem:q_in_01}
For every \(q \in (0,1)\), the diversity function \(D_q^K: P(X) \to
\R\) is continuous.
\end{lem}

\begin{proof}
First we break the space $X$ into convenient smaller pieces.  Put 
\[
b = \frac{1}{2} \inf_{x \in X} K(x,x) > 0.
\]
By the topological hypotheses on $(X, K)$, we can find
a finite open cover \(U_1,\ldots, U_n\) of \(X\) such that \(K(x,y) \geq
b\) whenever \(x,y \in \overline{U_i}\) for some \(i\), and we can find a
continuous partition of unity \(p_1, \ldots, p_n\) such that \(\supp p_i
\subseteq U_i\) for each \(i\).

For all \(\mu \in P(X)\),
\[
D_q^K(\mu)^{1-q} 
= 
\int_X (K\mu)^{q-1} \d\mu 
= 
\sum_{i=1}^n \int_X (K\mu)^{q-1} p_i \d\mu.
\]
Hence to see that \(D_q^K\) is continuous it will suffice to show that, for
each \(i\), the map $d_i: P(X) \to \R$ defined by
\[
d_i(\mu) = \int_X (K\mu)^{q - 1} p_i \d\mu
\]
is continuous. For the rest of the proof, fix \(i \in \{1,\ldots,n\}\).

For each $\mu \in P(X)$, we can bound $K\mu$ below on $\overline{U_i}$.
Indeed, for all \(x \in \overline{U_i}\),
\begin{align}
\label{eq:Kmu-lb}
(K\mu)(x)       
\geq 
\int_{U_i} K(x,y) p_i(y) \d\mu(y)       
\geq 
b \int_X p_i \d\mu.
\end{align}
Thus, $(K\mu)|_{\overline{U_i}}$ is bounded below by $b \int p_i \d\mu$,
which by Lemma~\ref{lem:disjoint_supps} is strictly positive if \(\supp p_i
\cap \supp \mu \neq \emptyset\).

Now we show that $d_i$ is continuous at each point in the set
\[
P_i(X) = \{\mu \in P(X) \st \supp p_i \cap \supp \mu \neq \emptyset\}.
\]
By Lemma~\ref{lem:disjoint_supps}, $P_i(X)$ is open in $P(X)$.  Thus, it is
equivalent to prove that the restriction of $d_i$ to $P_i(X)$ is continuous.

To do this, we begin by showing that there is a well-defined, continuous
map $G_i: P_i(X) \to C(\overline{U_i})$ given by
\[
G_i(\mu) = (K\mu)^{q - 1}|_{\overline{U_i}}.
\]
It is well-defined because, for each $\mu \in P_i(X)$, the map $K\mu$ is
continuous and strictly positive on $\overline{U_i}$ (as noted
after~\eqref{eq:Kmu-lb}).  To show that $G_i$ is continuous, consider the
following spaces and maps, defined below:
\[
P_i(X) \xto{K_*} 
C_i^+(X) \xto{\text{res}}
C^+(\overline{U_i}) \xto{(-)^{q - 1}}
C^+(\overline{U_i}) \hookrightarrow
C(\overline{U_i}).
\]
Here
\begin{align*}
C_i^+(X)        &
=
\{f \in C(X) \st 
f(x) > 0 \text{ for all } x \in \overline{U_i}\},       \\
C^+(\overline{U_i})     &
=
\{g \in C(\overline{U_i}) \st 
g(x) > 0 \text{ for all } x \in \overline{U_i}\} 
=
\cn{Top}(\overline{U_i}, (0, \infty)).
\end{align*}
The first map $K_*$ is the restriction of $K_*: P(X) \to C(X)$; the
restricted $K_*$ is well-defined by~\eqref{eq:Kmu-lb} and continuous by
Lemma~\ref{lem:K_*_cts}.  The second map is restriction, which is certainly
continuous, the third map $(-)^{q - 1}$ is continuous by
Lemma~\ref{lem:uniform_functorial}, and the last map is inclusion, also
continuous.  The composite of these maps is $G_i$, which is therefore also
continuous, as claimed.

To show that $d_i$ is continuous on $P_i(X)$, consider the chain of maps 
\[
P_i(X) 
\xto{\triangle} 
P_i(X) \times P(X) 
\xto{G_i \times (p_i \cdot -)} 
C(\overline{U_i}) \times P_\leq(\overline{U_i}) 
\xto{\langle -,-\rangle} 
\R
\]
(recalling the definition of $P_\leq$ from before
Lemma~\ref{lem:pairing_cts}).  The first map is the diagonal followed by an
inclusion; it is continuous. In the second, $p_i \cdot -$ is a restriction
of the map $M(X) \to M(\overline{U_i})$ defined by $\mu \mapsto p_i
\mu$, which is also continuous.  Since $G_i$ is continuous, so is $G_i
\times (p_i \cdot -)$.  The third map is continuous by
Lemma~\ref{lem:pairing_cts}(\ref{part:pairing_cts_2}).  And the composite
of the chain is $d_i|_{P_i(X)}$, which is, therefore, also continuous.

Finally, we show that the function \(d_i\) is continuous at all points
\(\mu \in P(X)\) such that \(\supp p_i \cap \supp \mu = \emptyset\). Fix
such a \(\mu\).

Given \(\nu\in P(X)\), if \(\supp p_i \cap \supp \nu = \emptyset\) then
\(d_i(\nu) = 0\), and otherwise
\[
d_i(\nu) 
= 
\int_{\overline{U_i}} (K\nu)^{q-1} p_i \d\nu 
\leq 
\int_{\overline{U_i}} \left( b \int_X p_i \d\nu \right)^{q-1} p_i \d\nu =
b^{q-1} \left( \int_X p_i \d\nu 
\right)^q
\]
(using the bound~\eqref{eq:Kmu-lb} and that \(q < 1\)). So in either case,
\begin{equation}
\label{eq:cts-triv}
0 \leq d_i(\nu) \leq b^{q-1} \left( \int_X p_i \d\nu\right)^q.
\end{equation}
Now as \(\nu \to \mu\) in \(P(X)\), we have
$
\int_X p_i \d\nu \to \int_X p_i \d\mu = 0,
$
so
\[
b^{q-1} \left( \int_X p_i \d\nu\right)^q \to 0 
\]
(since \(q > 0\)).  Hence the bounds~\eqref{eq:cts-triv} give
$d_i(\nu) \to 0 = d_i(\mu)$, as required.
\end{proof}

The remaining case of Proposition~\ref{prop:diversity_cts_mu}, $q = 1$,
will be deduced from the cases $q \in (0, 1)$ and $q \in (1, \infty)$.

\begin{lem}
The diversity function \(D_1^K: P(X) \to \R\) is continuous.
\end{lem}

\begin{proof}
Let \(\mu \in P(X)\) and \(\epsilon > 0\). Since $D_q^K(\mu)$ is continuous
in $q$ (Proposition~\ref{prop:diversity_cts_q}(\ref{part:dcq-cts})), we can
choose \(q^+ \in (1,\infty)\) such that
\[
|D_1^K(\mu) - D_{q^+}^K(\mu)| < \epsilon/2.
\]
Since \(D_{q^+}^K: P(X) \to \R\) is continuous, we can find a neighbourhood
\(U^+\) of \(\mu\) such that for all \(\nu \in U^+\),
\[
\bigl|D_{q^+}^K(\mu) - D_{q^+}^K(\nu)\bigr| < \epsilon/2.
\]
Then for all \(\nu \in U^+\), since $D_q^K(\nu)$ is decreasing in $q$
(Proposition~\ref{prop:diversity_cts_q}(\ref{part:dcq-dec})),  
\[
D_1^K(\nu) \geq D_{q^+}^K(\nu) > D_1^K(\mu) - \epsilon.
\]
Similarly, we can find a neighbourhood \(U^-\) of \(\mu\) such that for all
\(\nu \in U^-\),
\[
D_1^K(\nu) < D_1^K(\mu) + \epsilon 
\]
Hence $|D_1^K(\nu) - D_1^K(\mu)| < \epsilon$ for all $\nu \in U^+ \cap
U^-$. 
\end{proof}

This completes the proof of Proposition~\ref{prop:diversity_cts_mu}: the
diversity function of each finite positive order is continuous.

\begin{rmk}
Proposition~\ref{prop:diversity_cts_mu} excludes the cases $q = 0$ and $q =
\infty$.  Diversity of order $0$ is not continuous even in the simplest
case of a finite set and the identity similarity matrix; for as we saw in
Example~\ref{eg:fin-div-1}, $D_0^I(\vec{p})$ is the cardinality of $\supp
\vec{p}$, which is not continuous in $\vec{p}$.  Diversity of order
$\infty$ need not be continuous either.  For example, take $X = \{1, 2,
3\}$ and the similarity matrix
\[
Z = 
\begin{pmatrix}
1       &1      &0      \\
1       &1      &1      \\
0       &1      &1
\end{pmatrix},
\]
and put $\vec{p} = (1/2 - t, 2t, 1/2 - t)$.  Then
$D_\infty^Z(\vec{p})$ is $1$ if $t \in (0, 1/2)$, but $2$ if $t = 0$.
\end{rmk}

%%%%%%%%%%%%%%%%% MAGNITUDE %%%%%%%%%%%%%%%%%

\section{Magnitude}
\label{S_magnitude}

To show that maximum diversity and maximum entropy are well-defined, we
first have to define a closely related invariant, magnitude.  Magnitude has
been studied at various levels of generality, including finite enriched
categories and compact metric spaces, for which it has strong geometric
content~\cite{Leinstermagnitude2017}.  We will define the magnitude of a
space with similarities.

First we consider signed measures for which every point has typicality $1$.

\begin{defn}\label{def_weighting}
Let \(X = (X,K)\) be a space with similarities. A \demph{weight measure}
on \(X\) is a signed measure \(\mu \in M(X)\) such that $K\mu \equiv 1$ on
$X$.
\end{defn}

This generalises the definition of weight measure on a compact metric space
(Section~1.1 of \cite{Willertonmagnitude2014}).  Note that despite our
convention that `measure' means positive measure, a weight measure is a
\emph{signed} measure.

\begin{example}
\label{eg:mag-fin}
Let $X = \{1, \ldots, n\}$, writing $Z_{ij} = K(i, j)$ as usual.  Then a
weight measure on $X$ is a vector \(\vec{w} \in \R^n\) such that
\((Z\vec{w})_i = 1\) for \(i = 1, \ldots, n\).  If $Z$ is invertible then
there is exactly one weight measure, but in general there may be none or
many.

Even if $Z$ has many weight measures, the total weight $\sum_i w_i$ turns
out to be independent of the weighting $\vec{w}$ chosen, as long as $Z$ is
symmetric (or, more generally, the transpose of $Z$ admits a weighting
too).  This common quantity $\sum_i w_i$ is called the magnitude of $(X,
K)$, and its independence of the choice of weighting is a special case of
the following result.
\end{example}

%A space with similarities $(X, K)$ is \demph{symmetric} if $K$ is
%symmetric. 

\begin{lem}
\label{lem:wtg-ind}
Let $(X, K)$ be a symmetric space with similarities.  Then
$\mu(X) = \nu(X)$ for any weight measures $\mu$ and $\nu$ on $X$.
\end{lem}

\begin{proof}
Since $\nu$ is a weight measure,
\[
\mu(X) 
= 
\int_X \d\mu(x) 
= 
\int_X \left( \int_X K(x, y) \d\nu(y) \right) \d\mu(x).
\]
Since $\mu$ is a weight measure,
\[
\nu(X)
=
\int_X \d\nu(y)
=
\int_X \left( \int_X K(y, x) \d\mu(x) \right) \d\nu(y).
\]
So by symmetry of $K$ and Tonelli's theorem, $\mu(X) = \nu(X)$.
\end{proof}

This lemma makes the following definition valid.

\begin{defn}\label{def_magnitude}
Let \((X, K)\) be a symmetric space with similarities admitting at least
one weight measure. The \demph{magnitude} of \((X,K)\) is
\[
|(X,K)|= \mu(X),
\]
for any weight measure $\mu$ on $(X, K)$.  We often write $|(X, K)|$ as
just $|X|$.
\end{defn}

We will mostly use \emph{positive} weight measures, that is,
weight measures that are positive measures. (In an unfortunate clash of
terminology, a weight measure on a finite set is positive if and only if
the corresponding vector is nonnegative.)

\begin{lem}\label{lem:zero_magnitude}
Let $(X, K)$ be a symmetric space with similarities admitting a positive
weight measure.  Then $|X| \geq 0$, with equality if and only if $X =
\emptyset$. 
\end{lem}

\begin{proof}
The inequality is immediate from the definition of magnitude, as is the
fact that $|\emptyset| = 0$.  Now suppose that $X$ is nonempty.  Choose $x
\in X$ and a positive weight measure $\mu$ on \((X, K)\).  Since \(\int_X
K(x,-) \d\mu= 1\), the measure \(\mu\) is nonzero. Hence, \(|X| = \mu(X) >
0\).
\end{proof}

Let $(X, K)$ be a space with similarities.  Given a closed subset \(Y\) of
\(X\), we regard $Y$ as a space with similarities by restriction of the
similarity kernel \(K\).  Any measure $\nu \neq 0$ on $Y$ can be normalised
and extended by zero to give a probability measure \(\wext{\nu}\) on \(X\),
defined by
\[
\wext{\nu}(U) = \frac{\nu(U \cap Y)}{\nu(Y)}
\]
for Borel sets \(U \subseteq X\).  In particular, whenever $\nu$ is a
positive weight measure on $Y \neq \emptyset$, we have $\nu \neq 0$ (by
Lemma~\ref{lem:zero_magnitude}) and 
\[
\wext{\nu}(U) = \frac{\nu(U \cap Y)}{|Y|}
\]
for Borel sets \(U \subseteq X\).  The construction $\nu \mapsto
\wext{\nu}$ relates the notion of weight measure to that of balanced
measure (defined in Remark~\ref{rmk:typicality_constant}) as follows.

\begin{lem}\label{lem:balanced_tfae}
Let $(X, K)$ be a symmetric space with similarities.  The following are
equivalent for a probability measure $\mu$ on $X$:
\begin{enumerate}[(i)]
\item 
\label{part:it-const}
$\mu$ is balanced (that is, \(K\mu\) is constant on \(\supp \mu\));

\item 
\label{part:it-flat}
the function \(q \mapsto D_q^K(\mu)\) is constant on \([-\infty,
  \infty]\);

\item 
\label{part:it-supp}
\(\mu = \wext{\nu}\) for some positive weight measure \(\nu\) on
$\supp\mu$;

\item 
\label{part:it-wext}
\(\mu = \wext{\nu}\) for some positive weight measure \(\nu\) on
some nonempty closed subset $Y$ of $X$.
\end{enumerate}
When these conditions hold, \(D_q^K(\mu) = |Y|\) for all nonempty closed $Y
\subseteq X$ admitting a positive weight measure $\nu$ such that
$\wext{\nu} = \mu$, and all \(q \in [-\infty, \infty]\).
\end{lem}

\begin{proof}
The equivalence of~(\ref{part:it-const}) and~(\ref{part:it-flat}) follows from
Proposition~\ref{prop:diversity_cts_q}(\ref{part:dcq-dec}).

Now assuming~(\ref{part:it-const}), we prove~(\ref{part:it-supp}).  Write
$c$ for the constant value of $K\mu$ on $\supp\mu$.  Then $c > 0$ by
Proposition~\ref{prop:Kmu_props}(\ref{part:Kp-supp}), so we can define a
measure \(\nu\) on $\supp\mu$ by $\nu(W) = \mu(W)/c$ for all Borel sets $W
\subseteq \supp\mu$.  This \(\nu\) is a weight measure on \(\supp\mu\),
since for all \(y \in \supp\mu\),
\[
(K\nu)(y) 
= 
\int_{\supp\mu} K(y,-)\d\nu 
= 
\frac{1}{c} \int_X K(y,-) \d\mu
= 
\frac{1}{c} (K\mu)(y) 
=
1.
\]
Moreover, $\wext{\nu} = \mu$: for given a Borel set $U \subseteq X$,
\[
\wext{\nu}(U) 
= 
\frac{\nu(U \cap \supp\mu)}{\nu(\supp\mu)} 
= 
\frac{\mu(U \cap \supp\mu)}{\mu(\supp\mu)} 
= 
\mu(U),
\]
proving~(\ref{part:it-supp}).

Trivially, (\ref{part:it-supp}) implies~(\ref{part:it-wext}).  Finally, we
assume~(\ref{part:it-wext}) and prove~(\ref{part:it-const}).  Take $Y$ and
$\nu$ as in~(\ref{part:it-wext}).  For all \(x \in \supp\mu\),
\[
(K\mu)(x) 
= 
\int_X K(x,-) \d\wext{\nu}
= 
\frac{1}{\nu(Y)} \int_Y K(x,-) \d\nu 
= 
\frac{1}{\nu(Y)}
\]
This proves~(\ref{part:it-const}). It also proves the final statement: for
by Remark~\ref{rmk:typicality_constant},
$
D_q^K(\mu) = \nu(Y) = |Y|
$ 
for all \(q \in [-\infty, \infty]\).
\end{proof}

%%%%%%%%%%%%%%%%% BALANCED + MAXIMISING %%%%%%%%%%%%%%%%%

\section{Balanced and maximising measures}
\label{S_prep_lemmas}

In the case of the Kronecker delta on a finite discrete space, it is
trivial to maximise diversity.  Indeed, an elementary classical result
states that for each $q \in [0, \infty]$, the R\'enyi entropy $H_q^I$ of
order $q$ (Example~\ref{eg:fin-div-1}) is maximised by the uniform
distribution, and that unless $q = 0$, the uniform distribution is unique
with this property.  The same is therefore true of the diversity measures
$D_q^I$.

For a finite space with an arbitrary similarity kernel, maximising measures
are no longer uniform~\cite{LeinsterMaximizing2016}.  We cannot, therefore,
expect that on a general space with similarities, diversity is maximised by
the `uniform' measure (whatever that might mean).  Nevertheless, maximising
measures have a different uniformity property: they are balanced.  That is
the main result of this section.

\begin{rmk}
\label{rmk:q-range}
We usually restrict the parameter $q$ to lie in the range $[0, \infty]$.
Even in the simplest case of the Kronecker delta on a finite set, $D_q^K$
and $H_q^K$ behave quite differently for negative $q$ than for positive
$q$.  When $q < 0$, the uniform measure no longer maximises $D_q^I$ or
$H_q^I$, and in fact \emph{minimises} them among all measures of full
support (as can be shown using
Proposition~\ref{prop:diversity_cts_q}(\ref{part:dcq-dec})).
\end{rmk}

\glob{For the rest of this section, let \((X, K)\) be a symmetric space
  with similarities.}

\begin{defn}\label{def_maximising}
For \(q \in [0, \infty]\), a probability measure on \(X\) is
\demph{\(q\)-maximising} if it maximises \(D_q^K\). It is
\demph{maximising} if it is $q$-maximising for all \(q \in [0,\infty]\).
\end{defn}

We will show in Section~\ref{S_main} that any measure that is
$q$-maximising for some $q > 0$ is, in fact, maximising.  The proof will
depend on the next result: any measure that is $q$-maximising for some
$q \in (0, 1)$ is balanced.

This result can be understood as follows.  In ecological terminology, if a
species distribution is \emph{not} balanced then not all species are
equally typical, and it is intuitively plausible that transferring a little
abundance from the most typical species to the least typical increases
diversity.  Thus, the diversity of a non-balanced distribution should not
be maximal; equivalently, a distribution that maximises diversity should be
balanced.

We prove this using a variational argument.  The shape of the proof is
similar to that of the finite case (\cite{LeinsterMaximizing2016},
Section~5), but the generalisation to compact spaces makes the argument
much more delicate.

\begin{prop}\label{prop:qmax_balanced}
For \(q \in (0,1)\), every \(q\)-maximising measure on $(X, K)$ is balanced.
\end{prop}

\begin{proof}
Let \(q \in (0,1)\) and let \(\mu\) be a $q$-maximising measure on $(X, K)$.
Since \(K\mu\) is continuous and $\supp\mu$ is compact, we can choose
\(x^-, x^+ \in \supp \mu\) such that 
\[
(K\mu)(x^-) = \inf_{\supp \mu} K\mu,
\qquad
(K\mu)(x^+) = \sup_{\supp \mu} K\mu.
\]
To prove that $\mu$ is balanced, it will suffice to show that \((K\mu)(x^-)
= (K\mu)(x^+)\).

Let $\epsilon > 0$.  We first construct functions $u^\pm$ such that the
measures $u^\pm \mu$ approximate the Dirac measures at $x^\pm$, using
Lemma~\ref{lem:approximate_delta}.  Write
\[
E = 
\{ (K\mu)^{q - 1}|_{\supp \mu} \} 
\cup
\{ K(x, -)|_{\supp\mu} \st x \in X \} 
\subseteq
C(\supp\mu)
\]
(which is well-defined by Lemma~\ref{lem:Kmu_cts} and
Proposition~\ref{prop:Kmu_props}(\ref{part:Kp-supp})).  Then $E$ is
compact, since it is the union of a singleton with the image of the compact
space $X$ under the composite of continuous maps
\[
X \xto{\overline{K}} C(X) \xto{\text{restriction}} C(\supp\mu)
\]
(using Lemma~\ref{lem:Kbar_cts}).  Hence $E$ is equicontinuous.  So by
Lemma~\ref{lem:approximate_delta}, we can choose a nonnegative function
$u^- \in C(X)$ such that $\int_X u^- \d\mu = 1$ and
\begin{align*}
\left| \int_X (K\mu)^{q-1} \d(u^-\mu) - (K\mu)(x^-)^{q-1}\right| 
&
\leq \epsilon,     \\
\left| \int_X K(x,-) \d(u^-\mu) - K(x, x^-) \right| 
&
\leq \epsilon,
\end{align*}
the latter for all $x \in X$. Choose $u^+$ similarly for $x^+$.

Since \(u^- - u^+\) is bounded, we can choose an open interval \(I
\subseteq \R\), containing \(0\), such that the function \(1+ t\left(u^- -
u^+\right) \in C(X)\) is strictly positive for each \(t \in I\).  Then for
each \(t \in I\), we have a probability measure
\[
\mu_t = (1 + t(u^- - u^+)) \mu 
\]
on $X$, with $\supp\mu_t = \supp\mu$.  Note that $\mu_0 = \mu$.

We will exploit the fact that $D_q^K(\mu_t)$ has a local maximum at $t =
0$, showing that the function $t \mapsto D_q^K(\mu_t)^{1 - q}$ is
differentiable at $0$ and, therefore, has derivative $0$ there.  For each
$t \in I$,
\begin{align}
D_q^K(\mu_t)^{1 - q}    &
=
\int (K\mu_t)^{q - 1} \d\mu
+
t \int (K\mu_t)^{q - 1} d\bigl((u^- - u^+)\mu\bigr) 
\nonumber       \\
&
=
a(t) + b(t),
\label{eq:max_bal_0}      
\end{align}
say.  (Since $\supp(K\mu_t) \supseteq \supp(\mu_t) = \supp\mu$, the
integrand $(K\mu_t)^{q - 1}$ is well-defined and continuous on $\supp\mu$,
and both integrals are finite.)  We now show that $a(t)$ and $b(t)$ are
differentiable at $t = 0$, compute their derivatives there, and bound the
derivatives below.

To differentiate the integral $a(t)$, we use Lemma~\ref{lem:diff_int}.
Choose a bounded open subinterval $J$ of $I$, also containing $0$, with
$\overline{J} \subseteq I$.  We now verify that the function $f:
X \times J \to \R$ defined by
\[
f(x, t) 
= 
(K\mu_t)(x)^{q - 1}
=
\Bigl[ 
(K\mu)(x) + t K\bigl((u^- - u^+) \mu\bigr)(x)
\Bigr]^{q - 1}
\]
satisfies the conditions of Lemma~\ref{lem:diff_int}.

We have already checked condition~\ref{lem:diff_int}(\ref{part:di-int}).
For condition~\ref{lem:diff_int}(\ref{part:di-diff}): for all $x \in
\supp\mu$, the function \(f(x, -)\) is differentiable on $I$ (hence $J$),
with derivative
\[
t 
\mapsto 
\frac{\partial f}{\partial t} (x,t) 
= 
(q-1) 
\Bigl[
(K\mu)(x) + t K\bigl((u^- - u^+) \mu\bigr) (x)
\Bigr]^{q - 2}
\cdot
K\bigl((u^- - u^+)\mu\bigr)(x).
\]
For condition~\ref{lem:diff_int}(\ref{part:di-bound}), this formula shows
that $\partial f/\partial t$ is continuous on $(\supp\mu) \times I$.  Hence
$|\partial f/\partial t|$ is continuous on the compact space $(\supp\mu)
\times \overline{J}$, and therefore bounded on $(\supp\mu) \times J$, with
supremum $H$, say.  The constant function $H$ on $X$ is
$\mu$-integrable, and $\bigl|\tfrac{\partial f}{\partial t}(x, t)\bigr|
\leq H$ for all $x \in \supp \mu$ and $t \in J$, as required.

Now applying Lemma~\ref{lem:diff_int}, $a(t)$ is differentiable at $t = 0$
with
\begin{align}
a'(0)   &
= 
(q - 1) 
\int (K\mu)(x)^{q - 2} K\bigl((u^- - u^+)\mu\bigr)(x) \d\mu(x)        
\nonumber       \\
&
= 
(q-1) \int (K\mu)(x)^{q-2} 
\biggl( \int K(x, y) \d ((u^- - u^+) \mu)(y) \biggr)
\d\mu(x) 
\nonumber       \\
&
\geq
(q-1) 
\int (K\mu)^{q-2} \left( K(-, x^-) - K(-, x^+) + 2 \epsilon\right) \d\mu, 
\label{eq:a-bound}
\end{align}
where the inequality follows from the defining properties of \(u^-\) and
\(u^+\) and the fact that \(q < 1\).

Next, consider $b(t)$.  By definition of derivative, $b$ is differentiable
at $0$ if and only if the limit
\begin{align*}
% \label{eq:b-lim}
\lim_{t \to 0} 
\int (K\mu_t)^{q-1} \d((u^- - u^+)\mu)
\end{align*}
exists, and in that case $b'(0)$ is that limit.  As $t \to 0$, we have
$K\mu_t \to K\mu$ in $C(\supp\mu)$, so $(K\mu_t)^{q - 1} \to (K\mu)^{q -
  1}$ in $C(\supp\mu)$ (by Lemma~\ref{lem:uniform_functorial}).  Hence
$b'(0)$ exists and is given by
\[
b'(0) = \int_X (K\mu)^{q - 1}  \d\bigl( (u^- - u^+) \mu\bigr).
\]
By the defining properties of \(u^-\) and \(u^+\), it follows that
\begin{align}
\label{eq:b-bound}
b'(0)
\geq 
(K\mu)(x^-)^{q-1} - (K\mu)(x^+)^{q-1} - 2\epsilon.
\end{align}

Returning to equation~\eqref{eq:max_bal_0}, we have now shown that both $a(t)$
and $b(t)$ are differentiable at $t = 0$.  So too, therefore, is
$D_q^K(\mu_t)^{1 - q}$.  But by the maximality of $\mu$, its derivative
there is 0.  Hence the bounds~\eqref{eq:a-bound} and~\eqref{eq:b-bound}
give
\begin{align}
0 &
\geq
(q-1) 
\int (K\mu)^{q-2} \left( K(-, x^-) - K(-, x^+) + 2\epsilon \right) \d\mu
+ 
(K\mu)(x^-)^{q-1} - (K\mu)(x^+)^{q-1} - 2\epsilon 
\nonumber       \\ 
&
= 
(q-1) \left( 
\int (K\mu)^{q-2} K(x^-, -) \d\mu 
- 
\int (K\mu)^{q-2} K(x^+, -) \d\mu 
+ 
2\epsilon \int (K\mu)^{q - 2} \d\mu 
\right) \nonumber \\
& 
% \:\:\:\:\: {}
\phantom{= \mbox{}}
\mbox{} 
+ (K\mu)(x^-)^{q-1} - (K\mu)(x^+)^{q-1} - 2\epsilon,
\label{eq:ie_calc_1}
\end{align}
using the symmetry of $K$.  Consider the first integral
in~\eqref{eq:ie_calc_1}.  By definition of $x^-$, and since $q - 2 < 0$, we
have
\[
\int (K\mu)^{q - 2} K(x^-, -) \d\mu
\leq
(K\mu)(x^-)^{q - 2} \int K(x^-, -) \d\mu
=
(K\mu)(x^-)^{q - 1}.
\]
A similar statement holds for $x^+$.  Since $q - 1 < 0$, it follows
from~\eqref{eq:ie_calc_1} that
\begin{equation}
\label{eq:ie_calc_2}
0
\geq 
q \left( (K\mu)(x^-)^{q-1} - (K\mu)(x^+)^{q-1} \right) 
- 
2\epsilon \left( (1-q) \int (K\mu)^{q-2} \d\mu + 1 \right).
\end{equation}
Put \(c = (1-q) \int (K\mu)^{q-2} \d\mu + 1\).  Then
by~\eqref{eq:ie_calc_2}, the defining properties of \(x^-\) and \(x^+\),
and the fact that $0 < q < 1$,
\[
2\epsilon c 
\geq
q \left( (K\mu)(x^-)^{q-1} - (K\mu)(x^+)^{q-1} \right) 
\geq 
0.
\]
Taking \(\epsilon \to 0\), we see that \((K\mu)(x^-) = (K\mu)(x^+)\), which
proves the result.
\end{proof}

\begin{cor}\label{cor:qmax_balanced}
Assume that $X$ is nonempty.  For each \(q \in (0,1)\), there exists a
balanced \(q\)-maximising probability measure on~\(X\).
\end{cor}

\begin{proof}
Fix \(q \in (0,1)\). The function \(D_q^K\) is continuous on the nonempty
compact space \(P(X)\) (Proposition~\ref{prop:diversity_cts_mu}), so it
attains a maximum at some \(\mu \in P(X)\). By
Corollary~\ref{cor:qmax_balanced}, \(\mu\) is balanced.
\end{proof}

Thus, balanced $q$-maximising measures exist for arbitrarily small $q > 0$.
Later, we will use a limiting argument to find a balanced $0$-maximising
measure.  Any such measure maximises diversity of all orders
simultaneously:

\begin{lem}\label{lem:balanced_maximising}
For \(0 \leq q' \leq q \leq \infty\), any balanced probability measure
that is $q'$-maximising is also $q$-maximising. In particular, any
balanced $0$-maximising measure is maximising.
\end{lem}

\begin{proof}
Let \(0 \leq q' \leq q \leq \infty\) and let \(\mu\) be a balanced
$q'$-maximising measure.  Then for all \(\nu \in P(X)\),
\[
D_q^K(\nu) \leq D_{q'}^K(\nu) \leq D_{q'}^K(\mu) = D_q^K(\mu),
\]
where the inequalities follows from
Proposition~\ref{prop:diversity_cts_q}(\ref{part:dcq-dec}) and the
maximality of \(D_{q'}^K(\mu)\), and the equality from
Lemma~\ref{lem:balanced_tfae} and \(\mu\) being balanced.
\end{proof}

For the limiting argument, we will use:

\begin{lem}
\label{lem:closed}
\begin{enumerate}[(i)]
\item 
\label{part:closed-bal}
The set of balanced probability measures is closed in $P(X)$.

\item
\label{part:closed-max}
For each $q \in (0, \infty)$, the set of $q$-maximising probability
measures is closed in $P(X)$.
\end{enumerate}
\end{lem}

\begin{proof}
For~(\ref{part:closed-bal}), by Lemma~\ref{lem:balanced_tfae} and
Proposition~\ref{prop:diversity_cts_q}(\ref{part:dcq-dec}), the set of
balanced measures is
\[
\{ \mu \in P(X) \st D_1^K(\mu) = D_2^K(\mu) \}.
\]
But $D_1^K, D_2^K: P(X) \to \R$ are continuous (by
Proposition~\ref{prop:diversity_cts_mu}), so by a standard topological
argument, this set is closed.

Part~(\ref{part:closed-max}) is immediate from the continuity of $D_q^K$.
\end{proof}

%%%%%%%%%%%%%%%%% MAIN THEOREM %%%%%%%%%%%%%%%%%

\section{The maximisation theorem}
\label{S_main}

We now come to our main theorem:

\begin{thm}\label{thm:main}
Let \((X, K)\) be a nonempty symmetric space with similarities.  
\begin{enumerate}[(i)]
\item
\label{part:main-meas}
There exists a probability measure \(\mu\) on $X$ that maximises
\(D_q^K(\mu)\) for all \(q \in [0, \infty]\) simultaneously.

\item
\label{part:main-div}
The maximum diversity $\sup_{\mu \in P(X)} D_q^K(\mu)$ is independent of
\(q \in [0, \infty]\).
\end{enumerate}
\end{thm}

\begin{proof}
For each \(q \in (0,1)\), there exists a balanced \(q\)-maximising
probability measure on \(X\) (Corollary~\ref{cor:qmax_balanced}).  Since
\(P(X)\) is compact, we can choose some \(\mu \in P(X)\) such that for
every $q > 0$ and neighbourhood $U$ of $\mu$, there exist $q' \in (0, q)$
and a balanced $q'$-maximising measure in $U$.  Then by
Lemma~\ref{lem:balanced_maximising}, for every $q > 0$, every neighbourhood
of $\mu$ contains a balanced $q$-maximising measure.  To prove both parts
of the theorem, it suffices to show that $\mu$ is balanced and maximising.

By Lemma~\ref{lem:closed}(\ref{part:closed-bal}), $\mu$ is balanced.  By
Lemma~\ref{lem:closed}(\ref{part:closed-max}), $\mu$ is $q$-maximising for
each $q > 0$.  Now given any $\nu \in P(X)$, we have $D_q^K(\mu) \geq
D_q^K(\nu)$ for all $q > 0$; then passing to the limit as $q \to 0+$ and
using the continuity of diversity in its order
(Proposition~\ref{prop:diversity_cts_q}(\ref{part:dcq-cts})) gives
$D_0^K(\mu) \geq D_0^K(\nu)$.  Hence $\mu$ is $0$-maximising.  But $\mu$ is
also balanced, so by Lemma~\ref{lem:balanced_maximising}, $\mu$ is
maximising.
\end{proof}

The symmetry hypothesis in the theorem cannot be dropped, even in the
finite case (\cite{LeinsterMaximizing2016}, Section~8).

Part~(\ref{part:main-div}) of the theorem shows that maximum diversity is
an unambiguous real invariant of a space, not depending on a choice of
parameter~$q$:

\begin{defn}\label{def_maxdiv}
Let \((X, K)\) be a nonempty symmetric space with similarities.  The
\demph{maximum diversity} of \((X, K)\) is
\[
\Dmax{X}{K} = \sup_{\mu \in P(X)} D_q^K(\mu) \in (0, \infty),
\]
for any \(q \in [0,\infty]\).  
Similarly, the \demph{maximum entropy} of $(X, K)$ is
\[
\Hmax{X}{K} = \log\Dmax{X}{K} = \sup_{\mu \in P(X)} H_q^K(\mu).
\]
We often abbreviate $\Dmax{X}{K}$ as $\Dmx{X}$. 
\end{defn}

The well-definedness of maximum diversity can be understood as follows.  As
established in Section~\ref{S_prep_lemmas}, for a maximising measure $\mu$,
all points in $\supp\mu$ are equally typical.  Diversity is mean
atypicality, and although the notion of mean varies with the order $q$, all
means have the property that the mean of an essentially constant function
is that constant.  Thus, our maximising measure $\mu$ has the same
diversity of all orders.  That diversity is $\Dmx{X}$.

To find a measure that maximises diversity of \emph{all} positive orders,
it suffices to find one that maximises diversity of just \emph{one}
positive order:

\begin{cor}\label{cor:one_is_enough}
Let $(X, K)$ be a symmetric space with similarities.
Suppose that \(\mu \in P(X)\) is $q$-maximising for some \(q \in
(0,\infty]\). Then \(\mu\) is maximising.
\end{cor}

\begin{proof}
Fix \(q \in (0,\infty]\) and let \(\mu\) be a \(q\)-maximising measure. Then
\[
D_q^K(\mu) \leq D_0^K(\mu) \leq \Dmx{X} = D_q^K(\mu),
\]
so equality holds throughout.  As $D_q^K(\mu) = D_0^K(\mu)$ with $q \neq
0$, Proposition~\ref{prop:diversity_cts_q}(\ref{part:dcq-dec}) implies that
$\mu$ is balanced.  But also $D_0^K(\mu) = \Dmx{X}$, so $\mu$ is
$0$-maximising.  Lemma~\ref{lem:balanced_maximising} then implies that
$\mu$ is maximising.
\end{proof}

The exclusion of the case $q = 0$ here is necessary: not every
$0$-maximising measure is maximising, even in the finite case
(\cite{LeinsterMaximizing2016}, end of Section~6)

Theorem~\ref{thm:main} asserts the mere \emph{existence} of a maximising
measure and the well-definedness of maximum diversity.  But there is a
somewhat explicit \emph{description} of the maximum diversity and
maximising measures, in terms of magnitude and weight measures:

\begin{cor}
\label{cor:comp}
Let $(X, K)$ be a nonempty symmetric space with similarities.
\begin{enumerate}[(i)]
\item 
\label{part:comp-max}
We have
\begin{equation}
\label{eq:max-mag}
\Dmx{X} = \sup_Y |Y|,
\end{equation}
where the supremum is over the nonempty closed subsets $Y$ of $X$ admitting
a positive weight measure. 

\item
\label{part:comp-meas}
A probability measure $\mu$ on $X$ is maximising if and only if it is equal
to $\wext{\nu}$ for some positive weight measure $\nu$ on some subset
$Y$ attaining the supremum in~\eqref{eq:max-mag}.  In that case,
$\Dmx{X} = |\supp\mu|$.
\end{enumerate}
\end{cor}

\begin{proof}
For any $q
\in [0, \infty]$,
\begin{align}
\Dmx{X} &
=
\sup\{D_q^K(\mu) \st \mu \in P(X), \ \mu \text{ is balanced}\}
\label{eq:comp1}        \\
&
=
\sup\{|Y| \st \text{nonempty closed } Y \subseteq X \text{ admitting a
  positive weight measure}\}, 
\label{eq:comp2}        
\end{align}
where~\eqref{eq:comp1} follows from the existence of a balanced maximising
measure and~\eqref{eq:comp2} from Lemma~\ref{lem:balanced_tfae}.  This
proves~(\ref{part:comp-max}).  Every maximising measure is balanced,
so~(\ref{part:comp-meas}) also follows, again using
Lemma~\ref{lem:balanced_tfae}.  
\end{proof}

It follows that maximum diversity is monotone with
respect to inclusion:

\begin{cor}\label{cor:md_monotone_1}
Let \(X\) be a symmetric space with similarities, and let \(Y
\subseteq X\) be a nonempty closed subset. Then $\Dmx{Y} \leq \Dmx{X}$.
\qedhere
\end{cor}

Maximum diversity is also monotone in another sense: reducing the
similarity between points increases the maximum diversity.  For metric
spaces, this means that as distances increase, so does maximum
diversity. 

\begin{prop}\label{prop:md_monotone_2}
Let $X$ be a nonempty compact Hausdorff space. Let $K, K'$ be symmetric
similarity kernels on $X$ such that $K(x, y) \geq K'(x, y)$ for all $x, y
\in X$.  Then
$
\Dmax{X}{K} \leq \Dmax{X}{K'}.
$
\end{prop}

\begin{proof}
Fix $q \in [0, \infty]$.  We have $K\mu \geq K'\mu$ pointwise, so by 
definition of diversity, $D_q^K(\mu) \leq D_q^{K'}(\mu)$ for
all $\mu \in P(X)$.  Maximizing over $\mu$ gives the result.
\end{proof}

Maximising measures need not have full support.  Ecologically, that may
seem counterintuitive: can maximising diversity really entail eliminating
some species?  This phenonemon is discussed fully in Section~11
of~\cite{LeinsterMaximizing2016}, but in short: if a species is so ordinary
that all of its features are displayed more vividly by some other species,
then maximising diversity may indeed mean omitting it in favour of species
that are more distinctive.  With this in mind, it is to be expected that
any species absent from a maximising distribution is \hard{(i)}~at least as
ordinary or typical as those present, and \hard{(ii)}~reasonably similar to
at least one species present.  Since the typicality function of a
maximising measure $\mu$ takes constant value $1/\Dmx{X}$ on $\supp\mu$ (by
Proposition~\ref{prop:qmax_balanced}), this is the content of the following
lemma.

\begin{lem}
\label{lem:supertypical}
Let $\mu$ be a maximising measure on a nonempty symmetric space with
similarities $(X, K)$, and let $x \in X$.  Then:
\begin{enumerate}[(i)]
\item
\label{part:st-st}
$(K\mu)(x) \geq 1/\Dmx{X}$;

\item
\label{part:st-close}
there exists $y \in \supp\mu$ such that $K(x, y) \geq 1/\Dmx{X}$. 
\end{enumerate}
\end{lem}

The proof will use the symmetric bilinear form $\ip{-}{-}_X$ on $M(X)$
given by
\begin{equation}
\label{eq:form}
\ip{\nu}{\pi}_X
=
\int_X \int_X K(x, y) \d\nu(x) \d\pi(y),
\end{equation}
and the observation that $D_2^K(\nu) = 1/\ip{\nu}{\nu}_X$.  

\begin{proof}
To prove~(\ref{part:st-st}), for $s \in [0, 1]$, put
\[
\nu_s = (1 - s)\mu + s \delta_x \in P(X).
\]
Then for all $s \in [0, 1]$,
\begin{align*}
1/D_2^K(\nu_S) &
=
\Ip{(1 - s)\mu + s\delta_x}{(1 - s)\mu + s\delta_x}_X \\
&
=
(1 - s)^2 / \Dmx{X} + 2s(1 - s)\cdot (K\mu)(x) + s^2 K(x, x).
\end{align*}
Rearranging gives
\[
\frac{1}{D_2^K(\nu_S)} - \frac{1}{\Dmx{X}}      
=
\biggl\{
\biggl( \frac{1}{\Dmx{X}} - 2(K\mu)(x) + K(x, x) \biggr) s
+
2 \biggl( (K\mu)(x) - \frac{1}{\Dmx{X}} \biggr) 
\biggr\} s.
\]
But the left-hand side is nonnegative for all $s \in (0, 1]$, so the affine
function $\{\cdots\}$ of $s$ is nonnegative on $(0, 1]$, hence $(K\mu)(x)
- 1/\Dmx{X} \geq 0$.

To prove~(\ref{part:st-close}), it follows from~(\ref{part:st-st}) that 
\[
\frac{1}{\Dmx{X}}
\leq 
(K\mu)(x) 
=
\int_{\supp\mu} K(x, y) \d\mu(y)
\leq
\sup_{y \in \supp\mu} K(x, y),
\]
and since $\supp\mu$ is compact, the supremum is attained.
\end{proof}

%%%%%%%%%%%%%%%%% METRIC SPACES %%%%%%%%%%%%%%%%%

\section{Metric spaces}
\label{S_metric}

For the rest of this paper, we specialise to compact metric spaces $X = (X,
d)$, using the similarity kernel $K(x, y) = e^{-d(x, y)}$ and writing
$D_q^K$ as $D_q^X$.

We have seen that maximum diversity is closely related to magnitude
(Corollary~\ref{cor:comp}).  Here, we review some of the geometric
properties of magnitude (surveyed in~\cite{Leinstermagnitude2017}) and
state their consequences for maximum diversity.  We then compute maximum
diversity for several classes of metric space.

Most of the theory of the magnitude of metric spaces assumes that the space
is \demph{positive definite}, meaning that for every finite sequence $x_1,
\ldots, x_n$ of distinct points, the matrix $(e^{-d(x_i, x_j)})$ is
positive definite.  Many familiar metric spaces are positive definite,
including all subsets of $\R^n$ with the Euclidean or $\ell^1$ (taxicab)
metric, all subsets of hyperbolic space, and all ultrametric spaces
(\cite{MeckesPositive2013}, Theorem~3.6).

There are several equivalent definitions of the magnitude of a positive
definite compact metric space $X$, as shown by
Meckes~\cite{MeckesMagnitude2015,Leinstermagnitude2017}.  The simplest is
this:
\[
|X| = \sup \{ |Y| \st \text{finite } Y \subseteq X \}.
\]
When $X$ admits a weight measure (and in particular, when $X$ is finite),
this is equivalent to Definition~\ref{def_magnitude}.  Indeed, Meckes
proved (\cite{MeckesPositive2013}, Theorems~2.3 and~2.4): 

\begin{thm}[Meckes]
\label{thm:pdms_variational}
Let $X$ be a positive definite compact metric space.  Then
\[
|X| = 
\sup_\mu 
\frac{\mu(X)^2}{\int_X \int_X e^{-d(x, y)} \d\mu(x) \d\mu(y)},
\]
where the supremum is over all $\mu \in M(X)$ such that the denominator is
nonzero.  The supremum is attained by $\mu$ if and only if $\mu$ is a
scalar multiple of a weight measure, and if $\mu$ is a weight measure then
$|X| = \mu(X)$.  
\qedhere
\end{thm}

Note that the supremum is over \emph{signed} measures, unlike the similar
expression for maximum diversity in Example~\ref{eg:div2}.  Work such
as~\cite{Barcelomagnitudes2018} has established that even for some of the
most straightforward spaces (including Euclidean balls), no weight measure
exists.  In that case, the supremum is not attained.

An important property of positive definite spaces, immediate from the
definition, is that if $Y \subseteq X$ then $|Y| \leq |X|$.  Hence by
Corollary~\ref{cor:comp}(\ref{part:comp-max}),
\begin{equation}
\label{eq:max-leq-mag}
\Dmx{X} \leq |X|
\end{equation}
for all positive definite compact metric spaces $X \neq \emptyset$.  Any
one-point subset of $X$ has a positive weight measure and magnitude~$1$,
so again by Corollary~\ref{cor:comp}(\ref{part:comp-max}),
\[
\Dmx{X} \geq 1.
\]
If $X$ does not admit a weight measure then it follows from
Corollary~\ref{cor:comp}(\ref{part:comp-meas}) that no maximising measure
on $X$ has full support.  Indeed, the apparent rarity of spaces admitting a
weight measure suggests that the supremum in Corollary~\ref{cor:comp} runs
over a rather small class of subsets $Y$.

There are a few spaces of geometric interest whose magnitude is known
exactly, including spheres with the geodesic metric (Theorem~7
of~\cite{Willertonmagnitude2014}), Euclidean balls of odd dimension (whose
magnitude is a rational function of the radius \cite{Barcelomagnitudes2018,
  WillMOBH, Willertonmagnitude2018}), and convex bodies in
$\R^n$ with the $\ell^1$ metric (Theorem~5.4.6
of~\cite{Leinstermagnitude2017}; the magnitude is closely related to the
intrinsic volumes).  But for many very simple spaces, including even the
$2$-dimensional Euclidean disc, the magnitude remains unknown.

In the rest of this section, we analyse the few classes of metric space for
which we are able to calculate the maximum diversity exactly.  In principle
this includes all finite spaces, since Corollary~\ref{cor:comp} then
provides an algorithm for calculating the maximum diversity (described in
Section~7 of~\cite{LeinsterMaximizing2016}).  This class aside, all our
examples are instances of the following result.

\begin{lem}
\label{lem:pos-pos-dmax}
Let $X$ be a nonempty positive definite compact metric space admitting a
positive weight measure $\mu$.  Then:
\begin{enumerate}[(i)]
\item
\label{part:ppd-meas}
the normalisation $\wext{\mu}$ of $\mu$ is the unique maximising
measure on \(X\); 

\item
\label{part:ppd-dmax}
$\Dmx{X} = |X|$.
\end{enumerate}
\end{lem}

\begin{proof}
Since $X$ admits a positive weight measure,
Corollary~\ref{cor:comp}(\ref{part:comp-max}) gives $\Dmx{X} \geq |X|$.
But the opposite inequality~\eqref{eq:max-leq-mag} also holds, so $\Dmx{X}
= |X|$.  Hence by Corollary~\ref{cor:comp}(\ref{part:comp-meas}),
$\wext{\mu}$ is a maximising measure.  For uniqueness, let $\nu$ be any
maximising measure on $X$.  Then
\[
\frac{\nu(X)}{\int_X \int_X e^{-d(x, y)} \d\nu(x) \d\nu(y)}
=
D_2^X(\nu)
=
\Dmx{X}
=
|X|,
\]
so Theorem~\ref{thm:pdms_variational} implies that $\nu$ is a scalar multiple
of $\wext{\mu}$.  But both are probability measures, so $\nu =
\wext{\mu}$.  
\end{proof}

\begin{example}
\label{eg:scattered}
Let $X$ be a finite metric space with $n$ points, satisfying $d(x, y) >
\log(n - 1)$ whenever $x \neq y$.  Then $X$ is positive definite and its
unique weight measure is positive (Proposition~2.4.17
of~\cite{Leinstermagnitude2013}), so $\Dmx{X} = |X|$.
\end{example}

\begin{example}
\label{eg:ms-interval}
A line segment $[0,
  \ell] \subseteq \R$ has weight measure
\[
\tfrac{1}{2} (\delta_0 + \delta_\ell + \lambda_{[0, \ell]}),
\]
where $\delta_x$ denotes the Dirac measure at a point $x$ and $\lambda_{[0,
    \ell]}$ is Lebesgue measure on $[0, \ell]$
(\cite{Willertonmagnitude2014}, Theorem~2).  Hence
\[
|[0, \ell]| = 1 + \tfrac{1}{2}\ell.
\]
By Lemma~\ref{lem:pos-pos-dmax}, the maximum diversity of $[0, \ell]$ is
equal to its magnitude, and its unique maximising measure is
\[
\frac{\delta_0 + \delta_\ell + \lambda_{[0, \ell]}}{2 + \ell}.
\]
In fact, every compact subset of $\R$ has a positive weight measure (by
Lemma~2.8 and Corollary~2.10 of~\cite{MeckesPositive2013}), so again,
Lemma~\ref{lem:pos-pos-dmax} applies.
\end{example}

\begin{example}
\label{eg:ms-hgs}
Let $X$ be a nonempty compact metric space that is \demph{homogeneous}
(its isometry group acts transitively on points).  There is a unique
isometry-invariant probability measure on $X$, the Haar probability
  measure $\mu$ (Theorems~4.11 and~5.3 of~\cite{SteinlageHaar1975}).  As
observed in~\cite{Willertonmagnitude2014} (Theorem~1), the measure
\[
\frac{\mu}{\int_X e^{-d(x, y)} \d\mu(x)}
\]
is independent of $y \in X$ and is a positive weight measure on $X$.  Hence
\[
|X| = \frac{1}{\int_X e^{-d(x, y)} \d\mu(x)}
\]
for all $y \in X$.  This is the reciprocal of the expected similarity
between a random pair of points.  If $X$ is positive definite,
Lemma~\ref{lem:pos-pos-dmax} implies that $\Dmx{X} = |X|$ and the Haar
probability measure is the unique maximising measure.
\end{example}

We have shown that every symmetric space with similarities has at least one
maximising measure.  Although some spaces have multiple maximising measures
(\cite{LeinsterMaximizing2016}, Section~9), we now show that for many
metric spaces, the maximising measure is unique.

\begin{lem}
\label{lem:ip-unique}
Let $X$ be a nonempty compact metric space such that the bilinear form
$\ip{-}{-}_X$ on $M(X)$ (defined in~\eqref{eq:form}) is positive
definite. Then $X$ admits exactly one maximising measure.
\end{lem}

\begin{proof}
Since $\ip{-}{-}_X$ is an inner product, the function $\mu \mapsto
\ip{\mu}{\mu}_X$ on $M(X)$ is strictly convex.  Its restriction to the
convex set $P(X)$ therefore attains a minimum at most once.  But
$D_2^X(\mu) = 1/\ip{\mu}{\mu}_X$, so $\mu$ minimises $\ip{-}{-}_X$ on
$P(X)$ if and only if $\mu$ is 2-maximising, or equivalently maximising (by
Corollary~\ref{cor:one_is_enough}). The result follows.
\end{proof}

The next proposition follows immediately from Lemma~\ref{lem:ip-unique};
the subsequent more substantial result is due to Mark Meckes (personal
communication, 2019). 

\begin{prop}
Every nonempty positive definite finite metric space has exactly one
maximising measure.  
\end{prop}

\begin{prop}[Meckes]
\label{prop:euc-unique}
Every nonempty compact subset of Euclidean space has exactly one maximising
measure. 
\end{prop}

\begin{proof}
Let $X$ be a nonempty compact subset of $\R^n$.  Then $X$ is positive
definite, so by Lemma~2.2 of~\cite{MeckesPositive2013}, $\ip{\mu}{\mu}_X
\geq 0$ for all $\mu \in M(X)$.  By Lemma~\ref{lem:ip-unique}, it now
suffices to prove that if $\ip{\mu}{\mu}_X = 0$ then $\mu = 0$.

Define $F: \R^n \to \R$ by $F(x) = e^{-\|x\|}$.  Then 
\[
\ip{\mu}{\nu}_X = \int_{\R^n} (F * \mu) \d\nu
\]
($\mu, \nu \in M(X)$), where $*$ denotes convolution.  By the standard
properties of the Fourier transform $\hat{\ }$, it follows that 
\[
\ip{\mu}{\mu}_X = \int_{\R^n} \hat{F} |\hat{\mu}|^2 \d\lambda,
\]
where $\lambda$ is Lebesgue measure. But $\hat{F}$ is everywhere strictly
positive (Theorem~1.14 of~\cite{SteinIntroduction1971}), so if
$\ip{\mu}{\mu}_X = 0$ then $\hat{\mu} = 0$ almost everywhere, which in turn
implies that $\mu = 0$ (paragraph~1.7.3(b) of~\cite{RudinFourier1962}).
\end{proof}

\section{The uniform measure}
\label{S_uniform}

For many of the spaces that arise often in mathematics, there is a choice
of probability measure that seems natural to us.  For finite sets, it is
the uniform measure.  For homogeneous spaces, it is Haar measure.  For
subsets of $\R^n$ with finite nonzero volume, it is normalised Lebesgue
measure.  In this section, we propose a method for assigning a canonical
probability measure to any compact metric space (subject to conditions).
We call it the \emph{uniform measure}.

There are two thoughts behind this method.  The first is very standard in
statistics: take the probability distribution that maximises entropy.  For
example, in the context of differential entropy of probability distributions
on $\R$, the maximum entropy distribution supported on a prescribed bounded
interval is the uniform distribution on it, and the maximum entropy
distribution with a prescribed mean and variance is the normal
distribution.

However, on a compact metric space $X$, the maximising measure is in one
sense not a suitable choice of `uniform' measure.  The problem is
scale-invariance.  For many uses of metric spaces, the choice of scale
factor is somewhat arbitrary: if we multiplied all the distances by a
constant $t > 0$, we would regard the space as essentially unchanged.
(Formally, scaling by $t$ defines an automorphism of the category of metric
spaces, for any of the standard notions of map between metric spaces.)  But
the maximising measure depends critically on the scale factor, as almost
every example in the previous section shows.

There now enters the second thought: pass to the large-scale
limit.  Thus, we define the uniform measure on a space to be the limit of
the maximising measures as the scale factor increases to $\infty$.  Let us
now make this formal.

\begin{defn}
Let $X = (X, d)$ be a metric space and $t \in (0, \infty)$.  We write $td$
for the metric on $X$ defined by $(td)(x, y) = t\cdot d(x, y)$, and $K^t$
for the similarity kernel on $X$ defined by $K^t(x, y) = e^{-td(x, y)}$.
We denote by $tX$ the set $X$ equipped with the metric $td$.
\end{defn}

By Proposition~\ref{prop:md_monotone_2}, $\Dmx{tX}$ is increasing in $t \in
(0, \infty)$, for any compact metric space $X$.  If $X$ is a subspace of
$\R^n$ then $tX = (X, td)$ is isometric to $(\{tx \st x \in X\}, d)$, where
$d$ is Euclidean distance. But we will regard the set $X$ as fixed and the
metric as varying with $t$.

\begin{defn}
\label{defn:ufm}
Let $X$ be a compact metric space.  Suppose that $tX$ has a unique
maximising measure $\mu_t$ for all $t \gg 0$, and that $\lim_{t \to \infty}
\mu_t$ exists in $P(X)$.  Then the \demph{uniform measure} on $X$ is
$\mu_X = \lim_{t \to \infty} \mu_t$.  
\end{defn}

The uniform measure has the desired property of scale-invariance:

\begin{lem}
Let $X$ be a compact metric space and $t > 0$.  Then $\mu_X = \mu_{tX}$,
with one side of the equality defined if and only if the other is.
\end{lem}

\begin{proof}
This is immediate from the definition.
\end{proof}

The next few results show that on several significant classes of space, the
uniform measure is the canonical or `obvious' probability measure.

\begin{prop}
\label{prop:ufm-fin}
On a nonempty finite metric space, the uniform measure exists and is
the uniform probability measure in the standard sense.
\end{prop}

\begin{proof}
Let $X = \{x_1, \ldots, x_n\}$ be a finite metric space.  For $t > 0$, write
$Z^t$ for the $n \times n$ matrix with entries $e^{-td(x_i, x_j)}$.  For $t
\gg 0$, the space $tX$ is positive definite with positive weight measure,
by Example~\ref{eg:scattered}.  Expressed as a vector, the weight measure
on $tX$ (for $t \gg 0$) is
\[
(Z^t)^{-1} 
\begin{pmatrix} 
1 \\ \vdots \\ 1 
\end{pmatrix}.
\]
The normalisation of this weight measure is the unique maximising measure
$\mu_t$ on $tX$, by Lemma~\ref{lem:pos-pos-dmax}.  As $t \to \infty$, we
have $Z^t \to I$ in the topological group $\mathrm{GL}_n(\R)$, giving
$(Z^t)^{-1} \to I$ and so $\mu_t \to (1/n, \ldots, 1/n)$.
\end{proof}

This result shows that the uniform measure need not be uniformly
distributed, in the sense that balls of the same radius may have different
measures. 

The concept of uniform measure also behaves well on homogeneous spaces.  We
restrict to those spaces $X$ such that $tX$ is positive definite for every
$t > 0$, which is equivalent to the classical condition that $X$ is of
\demph{negative type}.  (The proof of equivalence is essentially due
to Schoenberg; see Theorem~3.3 of~\cite{MeckesPositive2013}.)

\begin{prop}
\label{prop:ufm-hgs}
On a nonempty, homogeneous, compact metric space of negative type,
the uniform measure exists and is the Haar probability
measure.
\end{prop}

\begin{proof}
Let $X$ be such a space.  The Haar probability measure $\mu$ on $X$ is the
unique isometry-invariant probability measure on $X$, so it is also the
Haar probability measure on $tX$ for every $t > 0$.  Hence by
Example~\ref{eg:ms-hgs}, $\mu_t = \mu$ for all $t$, and the result follows
trivially. 
\end{proof}

\begin{prop}
\label{prop:ufm-interval}
On the line segment $[0, \ell]$ of length $\ell > 0$, the uniform
measure exists and is Lebesgue measure restricted to $[0, \ell]$,
normalised to a probability measure.
\end{prop}

\begin{proof}
Write $X = [0, \ell]$ and $d$ for the metric on $\R$.  For each $t > 0$,
the metric space $tX = (X, td)$ is isometric to the interval $[0, t\ell]$
with metric $d$, which by Example~\ref{eg:ms-interval} has unique
maximising measure
\[
\frac{\delta_0 + \delta_{t\ell} + \lambda_{[0, t\ell]}}{2 + t\ell}.
\]
Transferring this measure across the isometry, $tX$ therefore has unique
maximising measure
\[
\mu_t = 
\frac{\delta_0 + \delta_\ell + t\lambda_{[0, \ell]}}{2 + t\ell}.
\]
Hence $\mu_t \to \lambda_{[0, \ell]}/\ell$ as $t \to \infty$.
\end{proof}

We now embark on the proof that Proposition~\ref{prop:ufm-interval} extends
to higher dimensions.  Let $X$ be a compact subspace of $\R^n$ with nonzero
volume, write $\lambda_X$ for $n$-dimensional Lebesgue measure $\lambda$
restricted to $X$, and write $\wext{\lambda_X} = \lambda_X/\lambda(X)$ for
its normalisation to a probability measure.  We will show that
$\wext{\lambda_X}$ is the uniform measure on $X$.  Unlike in
Propositions~\ref{prop:ufm-fin}--\ref{prop:ufm-interval}, we have no
formula for the maximising measure on $tX$, so the argument is 
less direct.

We begin by showing that at large scales, $\wext{\lambda_X}$ comes close to
maximising diversity, in the sense of the last part of the following
proposition.

\begin{prop}
\label{prop:preufm-euc}
Let $X$ be a compact subspace of $\R^n$ with nonzero volume $\lambda(X)$.
Then
\[
\lim_{t \to \infty} \frac{\Dmx{tX}}{|tX|} = 1
\qquad \text{and} \qquad
\lim_{t \to \infty} \frac{\Dmx{tX}}{t^n} = \frac{\lambda(X)}{n!\omega_n},
\]
where $\omega_n$ is the volume of the unit ball in $\R^n$.  Moreover, for
all $q \in [0, \infty]$,
\[
\lim_{t \to \infty} \frac{D_q^{tX}(\wext{\lambda_X})}{\Dmx{tX}} 
=
1.
\]
\end{prop}

\begin{proof}
We first show that for all $t > 0$ and $q \in [0, \infty]$,
\begin{equation}
\label{eq:ue1}
|tX| \geq
\Dmx{tX} \geq
D_q^{tX}(\wext{\lambda_X}) \geq
\frac{\lambda(X)t^n}{n!\omega_n}.
\end{equation}
The first inequality in~\eqref{eq:ue1} is an instance
of~\eqref{eq:max-leq-mag}, since $tX$ is positive definite.  The second
is immediate.  For the third, diversity is decreasing in its order,
% (Proposition~\ref{prop:diversity_cts_q}(\ref{part:dcq-dec})), 
so it suffices to prove the case $q = \infty$.  The inequality then states
that
\[
\frac{1}{\sup_{x \in \supp\wext{\lambda_X}} (K^t \wext{\lambda_X})(x)}
\geq
\frac{\lambda(X)t^n}{n!\omega_n},
\]
or equivalently, for all $x \in \supp\wext{\lambda_X}$,
\begin{equation}
\label{eq:ue2}
(K^t \wext{\lambda_X})(x)
\leq
\frac{n!\omega_n}{\lambda(X)t^n}.
\end{equation}
Now for all $x \in X$,
\[
(K^t \wext{\lambda_X})(x)       
=
\frac{1}{\lambda(X)} \int_X e^{-t\|x - y\|} \d y
\leq
\frac{1}{\lambda(X)} \int_{\R^n} e^{-t\|x - y\|} \d y.
\]
The last integral is $n!\omega_n/t^n$, by a standard calculation (as in
Lemma~3.5.9 of~\cite{Leinstermagnitude2013}).  So we have now proved
inequality~\eqref{eq:ue2} and, therefore, all of~\eqref{eq:ue1}.

Dividing~\eqref{eq:ue1} through by $|tX|$ gives
\[
1 \geq
\frac{\Dmx{tX}}{|tX|} \geq
\frac{D_q^{tX}(\wext{\lambda_X})}{|tX|} \geq
\frac{\lambda(X)t^n}{n! \omega_n |tX|}
\]
for all $t > 0$ and $q \in [0, \infty]$.  Theorem~1
of~\cite{Barcelomagnitudes2018} states, in part, that the final term
converges to $1$ as $t \to \infty$.  Hence all terms do, and the result
follows.
\end{proof}

\begin{rmks}
\label{rmks:Dmax-euc}
\begin{enumerate}[(i)]
\item
\label{rmk:Dmax-euc-cap}
The fact that $\Dmx{X}/|tX| \to 1$ as $t \to \infty$ is one of a collection
of results expressing the relationship between maximum diversity and
magnitude.  Perhaps the deepest of these is a result of Meckes, who showed
that maximum diversity is equal to a quantity that is already known (if
little explored) in potential theory: up to a constant, $\Dmx{X}$ is the
Bessel capacity of order $(n + 1)/2$ of $X$ (\cite{MeckesMagnitude2015},
Section~6).  He used this fact to prove that for each $n \geq 1$, there is
a constant $\kappa_n$ such that
\[
|X| \leq \kappa_n \Dmx{X}
\]
for all nonempty compact $X \subseteq \R^n$ (Corollary~6.2
of~\cite{MeckesMagnitude2015}).  This is a companion to the elementary
fact that $\Dmx{X} \leq |X|$ (inequality~(\ref{eq:max-leq-mag})).

\item
The second equation in Proposition~\ref{prop:preufm-euc} implies that the
volume of $X \subseteq \R^n$ can be recovered from the function $t \mapsto
\Dmx{tX}$.  This result is in the same spirit as
Theorem~\ref{thm:minkowski}, which states that one can also recover the
Minkowski dimension of $X$ from the asymptotics of $\Dmx{tX}$.
\end{enumerate}
\end{rmks}

\begin{thm}
\label{thm:ufm-euc}
On a compact set $X \subseteq \R^n$ of nonzero Lebesgue measure, the
uniform measure exists and is equal to Lebesgue measure restricted to $X$,
normalised to a probability measure.
\end{thm}

\begin{proof}
By Proposition~\ref{prop:euc-unique}, $tX$ has a unique maximizing measure
$\mu_t$ for each $t > 0$.  We must prove that $\lim_{t \to \infty} \int_X
f \d\mu_t = \int_X f\d\wext{\lambda_X}$ for each $f \in C(X)$.

Define $F \in C(\R^n)$ by $F(x) = e^{-\|x\|}$; then $\int_{\R^n} F
\d\lambda = n!\omega_n$, as noted in the proof of
Proposition~\ref{prop:preufm-euc}.  We will apply
Lemma~\ref{lem:approx-conv} to the function $G = F/n!\omega_n$.  We have
$G_t = t^n F^t/n!\omega_n$ for $t > 0$, and $\int_{\R^n} G \d\lambda = 1$.

First we prove the weaker statement that for all nonnegative $f \in C(X)$, 
\begin{equation}
\label{eq:liminf}
\liminf_{t \to \infty} \int_X f \d\mu_t 
\geq 
\int_X f \d\wext{\lambda_X}.
\end{equation}
Fix $f$, and choose a nonnegative extension $\bar{f} \in C(\R^n)$ of
bounded support.  Let $\epsilon > 0$.  By Lemma~\ref{lem:approx-conv}, we
can choose $T_1 > 0$ such that for all $t \geq T_1$,
\[
\int_{\R^n} 
\bar{f} \cdot \biggl( \frac{t^n F^t}{n!\omega_n} * \mu_t \biggr) 
\d\lambda
-
\int_{\R^n} \bar{f} \d\mu_t
\leq 
\frac{\epsilon}{2}.
\]
By Proposition~\ref{prop:preufm-euc}, we can also choose $T_2 > 0$ such
that for all $t \geq T_2$,
\[
\frac{t^n/n!\omega_n}{\Dmx{tX}} 
\geq 
\frac{1}{\lambda(X)} -  
\frac{\epsilon}{2\int_X f \d\lambda}.
\]
Then for all $t \geq \max\{T_1, T_2\}$,
\begin{align}
\int_X f \d\mu_t      &
=
\int_{\R^n} \bar{f} \d\mu_t   
\label{eq:euc1}    \\
&
\geq
\int_{\R^n} 
\bar{f} \cdot \biggl( \frac{t^n F^t}{n!\omega_n} * \mu_t \biggr) 
\d\lambda
- \frac{\epsilon}{2}
\label{eq:euc2}    \\
&
\geq
\int_X
f \cdot \biggl( \frac{t^n F^t}{n!\omega_n} * \mu_t \biggr) 
\d\lambda
- \frac{\epsilon}{2}
\label{eq:euc3}    \\
&
=
\int_X f \cdot \frac{t^n}{n!\omega_n} (K^t \mu_t) \d\lambda 
- \frac{\epsilon}{2}
\label{eq:euc4}    \\
&
\geq
\int_X f \cdot \frac{t^n/n!\omega_n}{\Dmx{tX}} \d\lambda
- \frac{\epsilon}{2}
\label{eq:euc5}    \\
&
\geq
\int_X f \d\wext{\lambda_X}
- \epsilon,
\label{eq:euc6}
\end{align}
where~\eqref{eq:euc1} holds because $\mu_t$ is supported on $X$,
\eqref{eq:euc2} because $t \geq T_1$, \eqref{eq:euc3}~because $\bar{f}$,
$F^t$ and $\mu_t$ are nonnegative, \eqref{eq:euc4}~because $F^t * \mu_t =
K^t\mu_t$, \eqref{eq:euc5}~by
Lemma~\ref{lem:supertypical}(\ref{part:st-st}), and~\eqref{eq:euc6} because
$t \geq T_2$ and $f \geq 0$.  The claimed inequality~\eqref{eq:liminf}
follows.

Now observe that if $f \in C(X)$ satisfies~\eqref{eq:liminf} then so does
$f + c$ for all constants~$c$.  But every function in $C(X)$ can be
expressed as the sum of a continuous nonnegative function and a constant,
so~\eqref{eq:liminf} holds for all $f \in C(X)$.  Let $f \in C(X)$.
Applying~\eqref{eq:liminf} to $-f$ in place of $f$ gives
\[
\limsup_{t \to \infty} \int_X f \d\mu_t
\leq
\int_X f \d\wext{\lambda_X},
\]
which together with~\eqref{eq:liminf} itself gives the desired result.  
\end{proof}

\begin{rmk}
Let $X \subseteq \R^n$ be a compact set of nonzero volume.  Then
$\supp\mu_t \to X$ in the Hausdorff metric $\dH$ as $t \to \infty$.
Indeed, Lemma~\ref{lem:supertypical}(\ref{part:st-close}) applied to the
similarity kernel $K^t$ gives $t \dH(X, \supp \mu_t) \leq \Hmx{tX}$, so
\[
\dH(X, \supp\mu_t)
\leq
\frac{\Hmx{tX}}{t}
=
\frac{\Hmx{tX}}{\log t} \cdot \frac{\log t}{t}
\to
n \cdot 0
=
0
\]
as $t \to \infty$, by Theorem~\ref{thm:minkowski}.  (The same argument
applies to any compact metric space of finite Minkowski dimension.)

However, the support of the uniform measure $\wext{\lambda_X} =
\lim_{t \to \infty} \mu_t$ need not be $X$; that is, some nonempty open
sets may have measure zero.  Any nontrivial union of an $n$-dimensional set
with a lower-dimensional set gives an example.
\end{rmk}

%%%%%%%%%%%%%%%%% OPEN PROBLEMS %%%%%%%%%%%%%%%%%

\section{Open questions}\label{S_conjectures}

(1)~As a numerical invariant of compact metric spaces (and more generally,
of symmetric spaces with similarity), how does maximum diversity behave with
respect to products, unions, etc., of spaces?  What are the maximising
measures on a product or union of spaces, and what is the uniform measure?

(2)~Beyond the finite case, very few examples of maximising measures are
known.  What, for instance, is the maximising measure on a Euclidean ball or
cube?  We do not even know its support.  In the case of a Euclidean ball,
we conjecture that the support of the maximising measure is a finite union
of concentric spheres, the number of spheres depending on the radius.

(3)~The uniform measure, when defined, is a canonical probability measure
on a given metric space. But so too is the Hausdorff measure.  More
exactly, if the Hausdorff dimension $d$ of $X$ is finite then we have the
Hausdorff measure $\mathcal{H}^d$ on $X$, which if $0 < \mathcal{H}^d(X) <
\infty$ can be normalised to a probability measure.  What is the
relationship between the Hausdorff probability measure and the uniform
measure?  It is probably not simple: for example, on $\{1, 1/2, 1/3,
\ldots, 0\} \subseteq \R$, the uniform measure is well-defined (it is
$\delta_0$), but the Hausdorff probability measure is not.  

(4)~What is the relationship between our notion of the uniform measure on a
compact metric space and that proposed by Ostrovsky and
Sirota~\cite{OstrovskyUniform2014} (based on entropy of a
different kind)?

(5)~For finite spaces with similarity, the diversity measures
$D_q^K$ were first introduced in ecology~\cite{LeinsterMeasuring2012} and
have been successfully applied there.  What are the biological applications
of our diversity measures on non-finite spaces?  In particular, in 
microbial biology it is common to treat the space of possible organisms as
a continuum.  Sometimes groupings are created, such as serotypes (strains)
of a virus or operational taxonomic units (genetically similar classes) of
bacteria, but it is recognised that these can be artificial.  What
biological information do our diversity measures
convey about continuous spaces of organisms?

%%%%%%%%%%%%%%%%% BIBLIOGRAPHY %%%%%%%%%%%%%%%%%

%\printbibliography
% \bibliography{max_div}

\end{document}